\documentclass[a4paper,preprint,11pt]{elsarticle}

\usepackage{amssymb}
\usepackage{algorithm}
\usepackage{algpseudocode}
\usepackage{color}
\usepackage[colorlinks=true, linkcolor=blue, citecolor=blue, urlcolor=blue]{hyperref}
\usepackage[normalem]{ulem}
\usepackage{bm}
\usepackage{tikz}
\usetikzlibrary{calc}
\usetikzlibrary{positioning}

\newcommand{\nlabel}[3]{%
  \node[#3=2pt] at (#1) {\small $#2$};%
}

\usepackage{graphicx}
\usepackage{array}
\usepackage{caption}
\usepackage{amsmath}

\journal{Information Processing Letter}

\newtheorem{theorem}{Theorem}

\newtheorem{lem}[theorem]{Lemma}
\newtheorem{remark}[theorem]{Remark}

\newtheorem{definition}[theorem]{Definition}

\newproof{proof}{Proof}
\algrenewcommand\algorithmicrequire{\textbf{Input:}}
\algrenewcommand\algorithmicensure{\textbf{Output:}}

\begin{document}
	
	\begin{frontmatter}
		
		
		
		\title{\textcolor{black}{Structural characterization and efficient recognition of probe diamond-free graphs}\tnoteref{t1}}
		
		\tnotetext[t1]{L.N.~Grippo, and M.C.~Lin}
		
		
		
		\author[UNGS,CONICET]{Luciano N. Grippo}
		\ead{lgrippo@campus.ungs.edu.ar}
		\author[IC,DC,CONICET]{Min Chih Lin}
		\ead{ oscarlin@dc.uba.ar }

		\address[UNGS]{Universidad Nacional de General Sarmiento, Instituto de Ciencias, Argentina}
		
		\address[IC]{Instituto de C\'alculo, Universidad Nacional de Buenos Aires, Argentina}
		
		\address[DC]{Departamento de Computaci\'on, Universidad de Buenos Aires, Argentina}
		
		\address[CONICET]{Consejo Nacional de Investigaciones Cient\'ificas y T\'ecnicas, Argentina}

		\begin{abstract}

A graph is probe diamond-free if its vertex set admits a partition into probes and nonprobes, where the set of nonprobes is independent, such that adding edges only between pairs of nonprobes yields a diamond-free graph.
Although this class admits a characterization by forbidden induced subgraphs,
such a characterization does not directly lead to an efficient recognition
algorithm. In this work we introduce a new structural characterization of probe
diamond-free graphs based on a local condition, called the
\emph{locally union of complete split} property, together with an auxiliary
bipartite graph.
Using this framework, we obtain an \(O(nm)\)-time recognition algorithm for
(nonpartitioned) probe diamond-free graphs.

A distinctive feature of our algorithm is that it is certificate-producing.
When the input graph does not belong to the class, the algorithm outputs a
negative certificate in the form of a sequence of vertices inducing a minimal
forbidden subgraph, ordered according to a fixed degree--lexicographic rule.
This ordered representation enables particularly simple and efficient
certificate verification.
When the input graph is probe diamond-free, the algorithm outputs a positive
certificate consisting of a probe partition and a completion set.

To the best of our knowledge, this is the first $O(nm)$-time recognition algorithm for probe diamond-free graphs that produces explicit certificates, providing an alternative to both sandwich-based approaches and exhaustive forbidden subgraph testing.
\end{abstract}

		
	\end{frontmatter}
	
	
	
\section{Introduction}\label{sec: intro}

Let $\mathcal G$ be a graph class. A graph $G=(V,E)$ is \emph{probe-$\mathcal G$} if its vertex set admits a partition into a set of probe vertices $P$ and an independent set of nonprobe vertices $N$ such that there exists a set $F$ of non-edges of $G$, each with both endpoints in $N$, for which the supergraph $G^*=(V,E\cup F)$ belongs to $\mathcal G$.

A graph $G=(P\cup N,E)$ is said to be \emph{partitioned} if its vertex set is given together with a partition into a set $P$ of probe vertices and an independent set $N$ of nonprobe vertices. A partitioned graph $G=(P\cup N,E)$ is a \emph{partitioned probe-$\mathcal G$ graph} if there exists a completion $G^*=(P\cup N,E\cup F)\in\mathcal G$ where every edge of $F$ has both endpoints in $N$. A probe-$\mathcal G$ graph with no prescribed probe partition is called a \emph{nonpartitioned} probe-$\mathcal G$ graph, \textcolor{black}{or simply a probe-$\mathcal G$ graph when the context is clear}.

We call $G^*$ a \emph{probe-$\mathcal G$ completion} of $G$; when no confusion may arise, we simply call it a \emph{completion}. We denote a probe partition of $G$ by \textcolor{black}{$(P,N)$}.

Given a family of graphs $\mathcal H$, we say that a graph $G$ is \emph{$\mathcal H$-free} if $G$ contains no member of $\mathcal H$ as an induced subgraph. When $\mathcal H=\{H\}$, we simply say that $G$ is \emph{$H$-free}.

Probe classes have been studied extensively in the literature. Probe interval graphs, in particular, have received considerable attention due to applications in biology~\cite{ZhangSchonFischerEtAl1994}. From an algorithmic standpoint, the fastest recognition algorithm for probe interval graphs with a given probe partition is due to McConnell and Nussbaum, who obtained a linear-time algorithm~\cite{McConnellN2009}. \textcolor{black}{In} the nonpartitioned setting, Chang et al. gave a polynomial-time recognition algorithm~\cite{ChangKloksLiuPeng2005}. Nussbaum later obtained a \textcolor{black}{linear-time algorithm for proper interval graphs} when the probe partition is part of the input~\cite{Nussbaum2014}.

Beyond probe interval graphs—the class that originally motivated the systematic study of probe-$\mathcal G$ classes—a body of work \textcolor{black}{has addressed} other probe classes from both algorithmic and structural perspectives. In 2007, Berry et al. presented an $O(m^2)$-time recognition algorithm for probe chordal graphs~\cite{BerryGL2007}. Le and Ridder obtained a linear-time recognition \textcolor{black}{algorithm} for probe cographs~\cite{LeR2007bis}, \textcolor{black}{as well as} an $O(n^2+nm)$-time recognition algorithm for probe split graphs~\cite{LeR2007}. Two years later, Bayer et al. gave \textcolor{black}{linear-time recognition algorithms} for both probe threshold and probe trivially perfect graphs~\cite{BAyerLR2009}. In 2013, Chang et al. proposed an $O(nm)$-time recognition algorithm for probe distance-hereditary graphs~\cite{Chang2013}. Le and Peng subsequently obtained a linear-time recognition algorithm for probe block graphs~\cite{LeP2015}. Probe permutation graphs~\cite{ChandlerCMK2009} and probe comparability graphs~\cite{ChandlerCKLP2008} have also been studied in the partitioned setting, where polynomial-time recognition algorithms are provided in both cases. Recently, Dabrowski et al. completely \textcolor{black}{determined} the complexity of $d$-\textsc{Cut} for every $d\ge 1$ on partitioned probe $H$-free graphs~\cite{DabrowskiEtAl2026}.

Deciding whether a given partitioned graph is probe-$\mathcal G$ can be viewed as a special case of the \emph{Graph Sandwich Problem}. In this problem, one is given a vertex set $V$ and two edge sets $E_1\subseteq E_2$, and asked whether there exists a graph \textcolor{black}{$H=(V,E_H)$} in $\mathcal G$ such that \textcolor{black}{$E_1\subseteq E_H\subseteq E_2$}. For $\mathcal G$ equal to the class of diamond-free graphs, Dantas et al.\textcolor{black}{~\cite{DantasFdaST2011}} showed that the corresponding sandwich problem can be solved in $O(n^4m)$ \textcolor{black}{time, where $n=|V|$ and $m=|E_2|$.
In our setting, $E_1=E$ and
\[
E_2 = E \cup \{xy \mid x,y\in N,\ x\neq y\},
\]
and hence $m=O(n^2)$, which yields an $O(n^6)$-time recognition algorithm for
partitioned probe diamond-free graphs.
}

The probe class obtained by taking $\mathcal G$ to be the class of diamond-free
graphs was characterized by a finite list of minimal forbidden induced subgraphs
in~\cite{BonomoFDDGSS}.
This yields a polynomial-time recognition algorithm for probe diamond-free
graphs; however, a direct implementation based on exhaustive forbidden-subgraph
testing runs in $O(n^9)$ time.

In this article, we present a new structural characterization of probe
diamond-free graphs based on a local condition—namely the
\emph{locally union of complete split} property—and an associated auxiliary
bipartite graph.
This characterization leads to an $O(nm)$-time recognition algorithm for
(nonpartitioned) probe diamond-free graphs.

A distinctive feature of our approach is that the algorithm is
\emph{certificate-producing}.
When the input graph does not belong to the class, the algorithm outputs a
negative certificate in the form of a \emph{sequence} of vertices inducing a
minimal forbidden subgraph, ordered according to a fixed
degree--lexicographic convention.
This contrasts with traditional forbidden-subgraph-based recognizers, which
typically output an unordered vertex set and may require exponential-time
verification over all permutations.

When the input graph is probe diamond-free, the algorithm outputs a positive
certificate consisting of a partition $(P,N)$ and a completion set $F$.
Although verifying such a positive certificate can be done by recognizing
diamond-freeness in the completed graph, the best known algorithms for this task
(e.g.~\cite{LinSS2012}) may lead to higher worst-case complexity due to the size
of $F$.
In this sense, our algorithm highlights a setting where producing certificates
is asymptotically easier than verifying them, which is uncommon in graph
recognition problems.

The paper is organized as follows.
In Section~\ref{sec: preliminaries} we introduce notation and recall basic facts
on probe classes and diamond-free graphs.
In Section~\ref{sec: characterization} we develop the structural framework
underlying our approach, including the locally union of complete split property
and the auxiliary bipartite graph $G_B$, and prove our new characterization of
probe diamond-free graphs.
Section~\ref{sec: algorithm} presents the recognition algorithm, establishes its
$O(nm)$ running time, and shows how negative certificates are produced.
We conclude in Section~\ref{sec: conclusion} with final remarks.

%

\section{Definitions and preliminary results}\label{sec: preliminaries}

Given a graph $G=(V,E)$ and a set $A\subseteq V$, we denote by $G[A]$ the subgraph of $G$ induced by $A$. By $C_n$ we denote a chordless cycle on $n$ vertices. The \emph{complement} of $G$, denoted by $\overline{G}$, is the graph with vertex set $V$ in which two distinct vertices $u,v\in V$ are adjacent if and only if $uv\notin E$. By $d_G(v)$ we denote the degree of a vertex $v$ in $G$. A \emph{clique} is a set of pairwise adjacent vertices\textcolor{black}{, and} the complete graph on $n$ vertices is denoted by $K_n$. By $N_G(v)$ we denote the neighborhood of $v$ in $G$ (when \textcolor{black}{the graph is clear from the context, we write} $N(v)$), and by $N_G[v]=N_G(v)\cup\{v\}$ the closed neighborhood of $v$. A \emph{diamond} is the graph $K_4-e$, \textcolor{black}{that is}, the complete graph on four vertices with one edge removed. The \emph{tips} of a diamond are the two vertices of degree two in the diamond, \textcolor{black}{equivalently, the endpoints of the missing edge}. Let $R$ and $S$ be two disjoint subsets of $V(G)$. We say that $R$ is \emph{complete to} $S$ if every vertex of $R$ is adjacent to every vertex of
$S$.

Let $G$ and $H$ be two graphs. The \emph{disjoint union} of $G$ and $H$, denoted by $G+H$, is the graph having $G$ and $H$ as its connected components.

Let $G=(P\cup N,E)$ and $H=(P'\cup N',E')$ be partitioned graphs with independent sets $N$ and $N'$, respectively. We say that $H$ is a \emph{partitioned subgraph} (respectively, a \emph{partitioned induced subgraph}) of $G$ if $H$ is a subgraph (respectively, an induced subgraph) of $G$ and $N'\subseteq N$ and $P'\subseteq P$. We say that $G$ is \emph{isomorphic} to $H$ if there exists a \textcolor{black}{bijective} function $f:P\cup N\to P'\cup N'$ preserving adjacency and \textcolor{black}{satisfying $f(N)=N'$ and $f(P)=P'$}. We say that $G$ \emph{does not contain} $H$ as \textcolor{black}{an} induced subgraph if no partitioned induced subgraph of $G$ is isomorphic to $H$.

A graph $G=(V,E)$ is a \emph{split graph} if its vertex set can be partitioned as \textcolor{black}{$V=K\cup S$}, where \textcolor{black}{$K$} is a clique and \textcolor{black}{$S$} is an independent set. Such a partition \textcolor{black}{$(K,S)$} is called a \emph{split partition} of $G$.

\textcolor{black}{Given a} split graph $G$ with a split partition \textcolor{black}{$V=K\cup S$, where $S$ is a maximum independent set, we say that $G$} is a \emph{complete split graph}\textcolor{black}{, and that $(K,S)$ is a \emph{complete split partition},}
if every vertex of \textcolor{black}{$K$} is adjacent to every vertex of \textcolor{black}{$S$}. 

It is well known that complete split graphs can be characterized by forbidden induced subgraphs as follows.

%

\begin{lem}\label{lem: complete_split_characterization}
	A graph $G$ is complete split if and only if $G$ is
	$\{\overline{P_3},\overline{2K_2}\}$-free.
\end{lem}

Equivalently, the two forbidden induced subgraphs are
$K_2+K_1=\overline{P_3}$ and $C_4=\overline{2K_2}$.
Moreover, every complete split graph $G$ that is not complete has a
unique complete split partition $(K,S)$.

The \emph{degree set} of a graph $G$ is the set of distinct degrees of its
vertices.

\textcolor{black}{
\begin{lem}\label{lem:complete_split_degree}
Let $G=(V,E)$ be a graph and let $n=|V|$.
Then $G$ is a complete split graph if and only if one of the following holds:
\begin{enumerate}
	\item $d(v)=n-1$ for every $v\in V$ (i.e., $G$ is complete);
	\item $d(v)=0$ for every $v\in V$ (i.e., $G$ has no edges);
	\item there exists an integer $k$ with $1\le k<n-1$ such that the degree set
	of $G$ is exactly $\{n-1,k\}$, and $k$ equals the number of vertices of degree
	$n-1$.
\end{enumerate}
Moreover, in Case~(2) the unique complete split partition is $(\emptyset,V)$,
and in Case~(3) the set
$K=\{v\in V : d(v)=n-1\}$ and $S=V\setminus K$
form the unique complete split partition of $G$.
\end{lem}
}

%
%

\begin{remark}\label{rem:complete_split_algoritm}
\textcolor{black}{
Lemma~\ref{lem:complete_split_degree} yields a very simple linear-time
recognition algorithm for complete split graphs, which we call the
\emph{Complete Split Degree Algorithm (CSDA)}.
Indeed, by computing the degree of every vertex and inspecting the degree set,
we can check whether one of the three cases of the lemma holds.
This requires $O(n+m)$ time, where $n=|V|$ and $m=|E|$, when the input is given
as the vertex and edge sets of the graph.
If the input consists only of the degree sequence, the running time can be
reduced to $O(n)$.
}

\textcolor{black}{
Moreover, whenever the recognition is positive, CSDA can also construct a
complete split partition $(K,S)$ within the same time bound:
in Case~(1) it outputs $(V\setminus\{v\},\{v\})$ for any vertex $v\in V$,
in Case~(2) it outputs $(\emptyset,V)$,
and in Case~(3) it sets
$K=\{v\in V : d(v)=n-1\}$ and $S=V\setminus K$.
}
\end{remark}

The following two results provide characterizations, in terms of minimal (partitioned) forbidden induced subgraphs, for a graph to be probe diamond-free.

\begin{theorem}~\cite{BonomoFDDGSS}\label{thm: charaterization partitioned}
Let $G=(P\cup N,E)$ be a partitioned graph. Then $G$ is a partitioned probe diamond-free graph if and only if $G$ does not contain any partitioned graph depicted in Figure~\ref{fig: forbidden partitioned subgraphs} as \textcolor{black}{an} induced subgraph. Moreover, a probe completion is \textcolor{black}{obtained} by adding \textcolor{black}{the set $F$ of non-edges} of $G$ whose endpoints belong to $N$ and \textcolor{black}{are tips} of the same induced diamond of $G$.
\end{theorem}

\begin{theorem}~\cite{BonomoFDDGSS}\label{thm: charaterization non-partitioned}
Let $G$ be a graph. Then $G$ is probe diamond-free if and only if $G$ is
$\{\overline{P_3 + 2K_1},\, \overline{P_4 + K_1},\, \overline{P_3 + K_2},\, \overline{2K_2 + K_1}\}$-free and does not contain any graph depicted in Figures~\ref{fig: forbidden subgraphs} and~\ref{fig: forbidden  subgraphs 2} as \textcolor{black}{an} induced subgraph. Moreover, a suitable probe partition is obtained by defining $N$ as the set of vertices of $G$ that are tips of an induced diamond, and $P=V\setminus N$. \textcolor{black}{A probe completion is then obtained} by adding the set $F$ of non-edges of $G$ whose endpoints are tips of the same induced diamond.
\end{theorem}

Let us also note that, in~\cite[Figure~3]{BonomoFDDGSS}, the graphs
$T_3$ and $T_7$ are isomorphic. This does not affect the stated number of
forbidden induced subgraphs, since one of them should be replaced by the graph
$T_4$ depicted in Figure~\ref{fig: forbidden  subgraphs}.
	
\section{Structural characterization of probe diamond-free graphs}\label{sec: characterization}
%
\subsection{Preliminaries}
A graph $G$ is \textcolor{black}{\emph{locally union of  complete split} (\emph{LUCS})} if, for every vertex $v\in V(G)$, each connected component of $G[N(v)]$ is a complete split graph. For each vertex $v \in V(G)$, let $C_v$ be a connected component of $G[N(v)]$ with \textcolor{black}{complete} split partition \textcolor{black}{$(K_{C_v},\, S_{C_v})$}. If \textcolor{black}{$|S_{C_v}| = 1$} we say that $C_v$ is a \emph{special component}\textcolor{black}{, that is,} $C_v$ is a complete graph. 

\textcolor{black}{
\begin{definition}\label{def:FG}
We define $F_G$ as the set of non-edges $uv$ of $G$ such that
$u,v\in S_{C_w}$ for some vertex $w\in V(G)$, where $C_w$ is a connected
component of $G[N(w)]$, and we set
\[
N_G=\bigcup_{uv\in F_G}\{u,v\}.
\]
\end{definition}
}

\textcolor{black}{
\begin{remark}
For every $uv\in F_G$, the vertices $w,u,v$ together with any vertex
of $K_{C_w}$ induce a diamond $D$ in $G$, where $u$ and $v$ are the tips of $D$.
\end{remark}
}
\begin{figure}
	\centering
	\includegraphics[scale = 0.7]{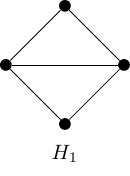}\quad \includegraphics[scale = 0.7]{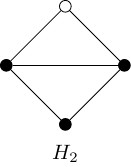}\quad \includegraphics[scale = 0.7]{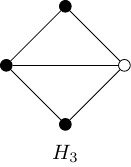}\quad \includegraphics[scale = 0.7]{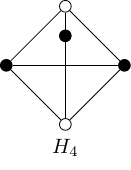}\quad \includegraphics[scale = 0.7]{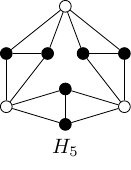}
	\caption{Minimal forbidden partitioned subgraphs for probe diamond-free graphs. Black vertices represent probes ($P$) and white vertices nonprobes ($N$).}
	\label{fig: forbidden partitioned subgraphs}
\end{figure}

By Lemma~\ref{lem: complete_split_characterization}, the following \textcolor{black}{result} holds.
\begin{lem}\label{lem: complete split}
A graph $G$ is \textcolor{black}{LUCS} if and only if it is $\{\overline{P_4 + K_1},\overline{P_3 + 2K_1},\overline{2K_2 + K_1}\}$-free.
\end{lem}
\begin{proof}
\textcolor{black}{
Recall that $G$ is LUCS if and only if, for every vertex $v\in V(G)$, each
connected component of $G[N(v)]$ is a complete split graph.
}

\textcolor{black}{
($\Rightarrow$) It follows immediately from Theorem~\ref{thm: charaterization non-partitioned}.
%
}

\textcolor{black}{
($\Leftarrow$) Conversely, assume that $G$ is
$\{\overline{P_4 + K_1},\overline{P_3 + 2K_1},\overline{2K_2 + K_1}\}$-free.
Fix a vertex $v\in V(G)$ and let $C$ be any connected component of $G[N(v)]$.
}
\textcolor{black}{
If $C$ contains an induced $K_2+K_1$, let $xy$ be the edge and $z$ the isolated
vertex (so $zx,zy\notin E(G)$). Since $C$ is connected, let
$P=z=p_0,p_1,\dots,p_\ell$ be a shortest path in $C$ from $z$ to $\{x,y\}$ and
assume $p_\ell=x$. Then $\ell\ge 2$ and $xp_{\ell-2}\notin E(G)$.
Let
}
\textcolor{black}{
\[
H=G[\{x,y,p_{\ell-2},p_{\ell-1}\}].
\]
}
\textcolor{black}{
By the minimality of $P$, we have $yp_{\ell-2}\notin E(G)$.
Hence $H$ induces either a $P_4$ or a paw, depending on whether
$yp_{\ell-1}\notin E(G)$ or $yp_{\ell-1}\in E(G)$.
Since $v$ is adjacent to every vertex of $C\subseteq N(v)$,
$G[\{v\}\cup V(H)]$ is isomorphic to $\overline{P_4+K_1}$ in the first case and
to $\overline{P_3+2K_1}$ in the second case, a contradiction.
}

\textcolor{black}{
If $C$ contains an induced $C_4$, then together with $v$ it yields an induced
$\overline{2K_2+K_1}$, a contradiction.
}

\textcolor{black}{
Hence $C$ is $\{K_2+K_1,\,C_4\}$-free, and by
Lemma~\ref{lem: complete_split_characterization} it follows that $C$ is complete
split. Since $v$ and $C$ were arbitrary, $G$ is LUCS.
}
\end{proof}

Notice that a diamond is a complete split graph. Consequently, if $v$ is a \textcolor{black}{vertex of degree three in an} induced diamond $D$ of a \textcolor{black}{LUCS} graph $G$, then $D-v$ is contained in a non-special component \textcolor{black}{$C_v$} of $G[N(v)]$\textcolor{black}{, and both tips of $D$ belong to $S_{C_v}$}.
\begin{remark}\label{rmk: tip of a dimond}
If $G$ is a \textcolor{black}{LUCS} graph, then the set of \textcolor{black}{vertices that are tips of induced diamonds of $G$ is exactly} $N_G$. \textcolor{black}{Moreover, if $G$ is probe diamond-free, then a probe completion of $G$ is obtained by adding the edge set $F_G$, as stated in Theorem~\ref{thm: charaterization non-partitioned}.}
\end{remark}

\tikzset{
  tnode/.style={circle, fill=black, inner sep=1.6pt},
  tedge/.style={line width=0.6pt}
}


\newcommand{\DrawBaseTwoDiamonds}{%
  \node[tnode] (t)  at (0,  1.6) {};
  \node[tnode] (l1) at (-1.2, 0.8) {};
  \node[tnode] (r1) at ( 1.2, 0.8) {};
  \node[tnode] (m1) at (0,  0.0) {};

  \node[tnode] (m2) at (0, -0.9) {};
  \node[tnode] (l2) at (-1.2,-1.7) {};
  \node[tnode] (r2) at ( 1.2,-1.7) {};
  \node[tnode] (b)  at (0, -2.5) {};

  \draw[tedge] (l1)--(r1);
  \draw[tedge] (t)--(l1);
  \draw[tedge] (t)--(r1);
  \draw[tedge] (m1)--(l1);
  \draw[tedge] (m1)--(r1);

  \draw[tedge] (m1)--(m2);

  \draw[tedge] (l2)--(r2);
  \draw[tedge] (m2)--(l2);
  \draw[tedge] (m2)--(r2);
  \draw[tedge] (b)--(l2);
  \draw[tedge] (b)--(r2);
}

\newcommand{\TOne}{%
\begin{tikzpicture}[scale=0.85]
  \DrawBaseTwoDiamonds
	\nlabel{t}{7}{left}	
	\nlabel{l1}{1}{above}
	\nlabel{r1}{2}{above}
	\nlabel{m1}{3}{left}
	\nlabel{m2}{4}{left}
	\nlabel{l2}{5}{below}
	\nlabel{r2}{6}{below}
	\nlabel{b}{8}{left}
\end{tikzpicture}}

\newcommand{\TTwo}{%
\begin{tikzpicture}[scale=0.85]
  \DrawBaseTwoDiamonds
  \draw[tedge] (l1)--(l2);
	\nlabel{t}{7}{left}	
	\nlabel{l1}{1}{above}
	\nlabel{r1}{3}{above}
	\nlabel{m1}{4}{left}
	\nlabel{m2}{5}{left}
	\nlabel{l2}{2}{below}
	\nlabel{r2}{6}{below}
	\nlabel{b}{8}{left}
\end{tikzpicture}}

\newcommand{\TThree}{%
\begin{tikzpicture}[scale=0.85]
  \DrawBaseTwoDiamonds
	\draw[bend left=25] (t) to (b);
	\nlabel{t}{4}{left}	
	\nlabel{l1}{1}{above}
	\nlabel{r1}{2}{above}
	\nlabel{m1}{3}{left}
	\nlabel{m2}{5}{left}
	\nlabel{l2}{7}{below}
	\nlabel{r2}{8}{below}
	\nlabel{b}{6}{left}	
\end{tikzpicture}}

\newcommand{\TFour}{%
\begin{tikzpicture}[scale=0.85]
  \DrawBaseTwoDiamonds
  \draw[tedge] (l1)--(b);
	\nlabel{t}{8}{left}	
	\nlabel{l1}{1}{above}
	\nlabel{r1}{2}{above}
	\nlabel{m1}{3}{left}
	\nlabel{m2}{5}{left}
	\nlabel{l2}{6}{below}
	\nlabel{r2}{7}{below}
	\nlabel{b}{4}{left}		
\end{tikzpicture}}

\newcommand{\TFive}{%
\begin{tikzpicture}[scale=0.85]
  \DrawBaseTwoDiamonds
	\draw[tedge] (l1)--(l2);
	\draw[tedge] (r1)--(r2);
	\nlabel{t}{7}{left}	
	\nlabel{l1}{1}{above}
	\nlabel{r1}{2}{above}
	\nlabel{m1}{5}{left}
	\nlabel{m2}{6}{left}
	\nlabel{l2}{3}{below}
	\nlabel{r2}{4}{below}
	\nlabel{b}{8}{left}		
\end{tikzpicture}}

\newcommand{\TSix}{%
\begin{tikzpicture}[scale=0.85]
  \DrawBaseTwoDiamonds
  \draw[tedge] (l1)--(l2);
	\draw[tedge] (r1)--(b);
	\nlabel{t}{8}{left}	
	\nlabel{l1}{1}{above}
	\nlabel{r1}{2}{above}
	\nlabel{m1}{4}{left}
	\nlabel{m2}{6}{left}
	\nlabel{l2}{3}{below}
	\nlabel{r2}{7}{below}
	\nlabel{b}{5}{left}		
\end{tikzpicture}}

\newcommand{\TSeven}{%
\begin{tikzpicture}[scale=0.85]
  \DrawBaseTwoDiamonds
  \draw[tedge] (l1)--(l2);
  \draw[bend left=25] (t) to (b);
	\nlabel{t}{5}{left}	
	\nlabel{l1}{1}{above}
	\nlabel{r1}{3}{above}
	\nlabel{m1}{4}{left}
	\nlabel{m2}{6}{left}
	\nlabel{l2}{2}{below}
	\nlabel{r2}{8}{below}
	\nlabel{b}{7}{left}	
\end{tikzpicture}}

\newcommand{\TEight}{%
\begin{tikzpicture}[scale=0.85]
  \DrawBaseTwoDiamonds
  \draw[tedge] (l1)--(b);
  \draw[tedge] (t)--(l2);
	\nlabel{t}{3}{left}	
	\nlabel{l1}{1}{above}
	\nlabel{r1}{5}{above}
	\nlabel{m1}{6}{left}
	\nlabel{m2}{8}{left}
	\nlabel{l2}{2}{below}
	\nlabel{r2}{7}{below}
	\nlabel{b}{4}{left}		
\end{tikzpicture}}

\newcommand{\TNine}{%
\begin{tikzpicture}[scale=0.85]
  \DrawBaseTwoDiamonds
  \draw[tedge] (l1)--(l2);
  \draw[tedge] (r1)--(r2);
  \draw[bend left=25] (t) to (b);	
	\nlabel{t}{6}{left}	
	\nlabel{l1}{1}{above}
	\nlabel{r1}{2}{above}
	\nlabel{m1}{5}{left}
	\nlabel{m2}{7}{left}
	\nlabel{l2}{3}{below}
	\nlabel{r2}{4}{below}
	\nlabel{b}{8}{left}			
\end{tikzpicture}}

\newcommand{\TTen}{%
\begin{tikzpicture}[scale=0.85]
  \DrawBaseTwoDiamonds
  \draw[tedge] (l1)--(l2);
  \draw[tedge] (t)--(r2);
  \draw[tedge] (b)--(r1);
	\nlabel{t}{5}{left}	
	\nlabel{l1}{1}{above}
	\nlabel{r1}{2}{above}
	\nlabel{m1}{6}{left}
	\nlabel{m2}{8}{left}
	\nlabel{l2}{3}{below}
	\nlabel{r2}{4}{below}
	\nlabel{b}{7}{left}		
\end{tikzpicture}}

\newcommand{\SOne}{%
\begin{tikzpicture}[scale=0.85]
\node[tnode] (a) at (-2,0) {};
\node[tnode] (b) at (-1,0.8) {};
\node[tnode] (c) at (-1,-0.8) {};
\node[tnode] (d) at (0,0) {};

\node[tnode] (e) at (1,0.8) {};
\node[tnode] (f) at (1,-0.8) {};
\node[tnode] (g) at (2,0) {};

\draw[tedge] (a)--(b)--(d)--(c)--(a);
\draw[tedge] (b)--(c);           
\draw[tedge] (d)--(g)--(e)--(d); 
\draw[tedge] (g)--(f)--(d);

\nlabel{d}{1}{above}	
\nlabel{b}{2}{left}	
\nlabel{c}{3}{left}
\nlabel{a}{7}{above}
\nlabel{g}{4}{above}
\nlabel{e}{5}{left}
\nlabel{f}{6}{left}
\end{tikzpicture}}

\newcommand{\STwo}{%
\begin{tikzpicture}[scale=0.85]
\node[tnode] (t)  at (0,1.35) {};
\node[tnode] (l)  at (-1.15,0.35) {};
\node[tnode] (r)  at (1.15,0.35) {};
\node[tnode] (bl) at (-0.65,-1.10) {};
\node[tnode] (br) at (0.65,-1.10) {};
\node[tnode] (il) at (-0.55,-0.10) {};
\node[tnode] (ir) at (0.55,-0.10) {};

\draw[tedge] (l)--(t)--(r);
\draw[tedge] (l)--(bl)--(br)--(r);
\draw[tedge] (l)--(il)--(bl);
\draw[tedge] (r)--(ir)--(br);
\draw[tedge] (t)--(il) (t)--(ir);

\nlabel{t}{1}{left}	
\nlabel{l}{2}{above}	
\nlabel{il}{3}{right}	
\nlabel{r}{4}{above}	
\nlabel{ir}{5}{left}	
\nlabel{bl}{6}{left}	
\nlabel{br}{7}{right}	

\end{tikzpicture}}

\newcommand{\SThree}{%
\begin{tikzpicture}[scale=0.85]
\node[tnode] (T)  at (0,1.70) {};
\node[tnode] (Lu) at (-1.35,0.75) {};
\node[tnode] (Ru) at (1.35,0.75) {};
\node[tnode] (Ll) at (-1.35,-0.55) {};
\node[tnode] (Rl) at (1.35,-0.55) {};
\node[tnode] (p)  at (-0.55,0.15) {};  
\node[tnode] (q)  at (0.55,0.15) {};   
\node[tnode] (C)  at (0,-0.10) {};     
\node[tnode] (B)  at (0,-0.95) {};     

\draw[tedge] (Lu)--(Ll) (Ru)--(Rl);
\draw[tedge] (T)--(Lu) (T)--(Ru);
\draw[tedge] (T)--(p) (T)--(q);
\draw[tedge] (p)--(Ll) (q)--(Rl);
\draw[tedge] (Ll)--(C)--(Rl);
\draw[tedge] (Ll)--(B)--(Rl);
\draw[tedge] (C)--(B);
\draw[tedge] (Lu)--(p) (Ru)--(q);

\nlabel{T}{1}{left}	
\nlabel{Ll}{2}{below}
\nlabel{Rl}{3}{below}	
\nlabel{Lu}{4}{above}	
\nlabel{p}{5\;}{above}	
\nlabel{Ru}{6}{above}	
\nlabel{q}{\;7}{above}	
\nlabel{C}{8}{above}	
\nlabel{B}{\;\;\;9}{above}	
\end{tikzpicture}}

\newcommand{\SFour}{%
\begin{tikzpicture}[scale=0.85]
\node[tnode] (top1) at (0,1.05) {};
\node[tnode] (bot1) at (0,-1.05) {};
\node[tnode] (L)   at (-1.10,0) {};
\node[tnode] (R)   at (1.10,0) {};
\node[tnode] (mid) at (0,0.35) {}; 

\draw[tedge] (top1)--(L)--(bot1)--(R)--(top1);
\draw[tedge] (L)--(R);      
\draw[tedge] (top1)--(mid)--(bot1); 

\nlabel{L}{2}{above}	
\nlabel{top1}{1}{left}
\nlabel{bot1}{4}{left}	
\nlabel{R}{3}{above}	
\nlabel{mid}{5}{left}	

\end{tikzpicture}}

\newcommand{\PFour}{%
\begin{tikzpicture}[scale=0.85]
\node[tnode] (a) at (0,0) {};
\node[tnode] (b) at (1,0) {};
\node[tnode] (c) at (2,0) {};
\node[tnode] (d) at (3,0) {};
\draw[tedge] (a)--(b)--(c)--(d);

\nlabel{a}{3}{above}	
\nlabel{b}{1}{above}
\nlabel{d}{4}{above}	
\nlabel{c}{2}{above}	
\end{tikzpicture}}

\newcommand{\Paw}{%
\begin{tikzpicture}[scale=0.85]
\node[tnode] (a) at (0,0) {};
\node[tnode] (b) at (1.2,0) {};
\node[tnode] (c) at (0.6,1.0) {};
\node[tnode] (p) at (-0.9,0) {}; 
\draw[tedge] (a)--(b)--(c)--(a) (a)--(p);

\nlabel{a}{1\;}{above}	
\nlabel{b}{\;3}{above}
\nlabel{p}{4}{above}	
\nlabel{c}{2}{left}
\end{tikzpicture}}

\newcommand{\CFour}{%
\begin{tikzpicture}[scale=0.85]
\node[tnode] (a) at (0,0) {};
\node[tnode] (b) at (1.2,0) {};
\node[tnode] (c) at (1.2,1.2) {};
\node[tnode] (d) at (0,1.2) {};
\draw[tedge] (a)--(b)--(c)--(d)--(a);

\nlabel{a}{1}{left}	
\nlabel{b}{3}{right}
\nlabel{d}{2}{left}	
\nlabel{c}{4}{right}	
\end{tikzpicture}}

\newcommand{\PFourUniversal}{%
\begin{tikzpicture}[scale=0.85]
\node[tnode] (a) at (0,0) {};
\node[tnode] (b) at (1,0) {};
\node[tnode] (c) at (2,0) {};
\node[tnode] (d) at (3,0) {};
\node[tnode] (u) at (1.5,1.2) {};
\draw[tedge] (a)--(b)--(c)--(d);
\draw[tedge] (u)--(a) (u)--(b) (u)--(c) (u)--(d);

\nlabel{a}{4}{above}	
\nlabel{b}{2\;}{above}
\nlabel{d}{5}{above}	
\nlabel{c}{\;3}{above}
\nlabel{u}{1}{above}
	
\end{tikzpicture}}

\newcommand{\PawUniversal}{%
\begin{tikzpicture}[scale=0.85]
\node[tnode] (a) at (0,0) {};
\node[tnode] (b) at (1.2,0) {};
\node[tnode] (c) at (0.6,1.0) {};
\node[tnode] (p) at (-0.9,0) {}; 
\node[tnode] (u) at (0.6,2.0) {}; 
\draw[tedge] (a)--(b)--(c)--(a) (a)--(p);
\draw[tedge] (u)--(a) (u)--(b) (u)--(c) (u)--(p);

\nlabel{a}{2\;\;}{above}	
\nlabel{b}{\;\;4}{above}
\nlabel{p}{5\;}{above}	
\nlabel{c}{3}{below}
\nlabel{u}{1}{right}
\end{tikzpicture}}

\newcommand{\CFourUniversal}{%
\begin{tikzpicture}[scale=0.85]
\node[tnode] (a) at (0,0) {};
\node[tnode] (b) at (1.2,0) {};
\node[tnode] (c) at (1.2,1.2) {};
\node[tnode] (d) at (0,1.2) {};
\node[tnode] (u) at (0.6,0.6) {};
\draw[tedge] (a)--(b)--(c)--(d)--(a);
\draw[tedge] (u)--(a) (u)--(b) (u)--(c) (u)--(d);

\nlabel{a}{2}{left}	
\nlabel{b}{4}{right}
\nlabel{d}{3}{left}	
\nlabel{c}{5}{right}
\nlabel{u}{1}{above}
	
\end{tikzpicture}}

\newcommand{\TOnePlus}{%
\begin{tikzpicture}[scale=0.85]
  \DrawBaseTwoDiamonds
	\nlabel{m1}{\textcolor{brown}{10}}{above}	
	\nlabel{m2}{\textcolor{brown}{10}}{below}	
	\nlabel{m1}{\textcolor{blue}{10}}{left}
	\nlabel{m2}{\textcolor{blue}{10}}{left}
\end{tikzpicture}}

\newcommand{\TTwoPlus}{%
\begin{tikzpicture}[scale=0.85]
  \DrawBaseTwoDiamonds
  \draw[tedge] (l1)--(l2);
	\nlabel{m1}{\textcolor{brown}{12}}{above}
	\nlabel{m2}{\textcolor{brown}{12}}{below}
  \nlabel{m1}{\textcolor{blue}{10}}{left}
	\nlabel{m2}{\textcolor{blue}{10}}{left}
	\nlabel{l1}{\textcolor{blue}{2}}{above}
	\nlabel{l2}{\textcolor{blue}{2}}{below}
\end{tikzpicture}}

\newcommand{\TThreePlus}{%
\begin{tikzpicture}[scale=0.85]
  \DrawBaseTwoDiamonds
	\draw[bend left=25] (t) to (b);
	\nlabel{m1}{\textcolor{brown}{20}}{above}	
	\nlabel{m2}{\textcolor{brown}{20}}{below}
	\nlabel{m1}{\textcolor{blue}{10}}{left}
	\nlabel{m2}{\textcolor{blue}{10}}{left}    
	\nlabel{t}{\textcolor{blue}{10}}{left}	
	\nlabel{b}{\textcolor{blue}{10}}{left}	
\end{tikzpicture}}

\newcommand{\TFourPlus}{%
\begin{tikzpicture}[scale=0.85]
  \DrawBaseTwoDiamonds
  \draw[tedge] (l1)--(b);
	\nlabel{m1}{\textcolor{brown}{16}}{above}
	\nlabel{m2}{\textcolor{brown}{17}}{below}
	\nlabel{m1}{\textcolor{blue}{10}}{left}
	\nlabel{m2}{\textcolor{blue}{10}}{right}  
	\nlabel{l1}{\textcolor{blue}{6}}{above}
	\nlabel{b}{\textcolor{blue}{7}}{left}		
\end{tikzpicture}}

\newcommand{\TFivePlus}{%
\begin{tikzpicture}[scale=0.85]
  \DrawBaseTwoDiamonds
	\draw[tedge] (l1)--(l2);
	\draw[tedge] (r1)--(r2);
	\nlabel{m1}{\textcolor{brown}{14}}{above}
	\nlabel{m2}{\textcolor{brown}{14}}{below}
	\nlabel{m1}{\textcolor{blue}{10}}{left}
	\nlabel{m2}{\textcolor{blue}{10}}{left}    
	\nlabel{l1}{\textcolor{blue}{2}}{above}
	\nlabel{r1}{\textcolor{blue}{2}}{above}
	\nlabel{l2}{\textcolor{blue}{2}}{below}
	\nlabel{r2}{\textcolor{blue}{2}}{below}
\end{tikzpicture}}

\newcommand{\TSixPlus}{%
\begin{tikzpicture}[scale=0.85]
  \DrawBaseTwoDiamonds
  \draw[tedge] (l1)--(l2);
	\draw[tedge] (r1)--(b);
	\nlabel{m1}{\textcolor{brown}{18}}{above}
	\nlabel{m2}{\textcolor{brown}{19}}{below}
	\nlabel{m1}{\textcolor{blue}{10}}{left}
	\nlabel{m2}{\textcolor{blue}{10}}{left}
	\nlabel{l1}{\textcolor{blue}{2}}{above}
	\nlabel{r1}{\textcolor{blue}{6}}{above}
	\nlabel{l2}{\textcolor{blue}{2}}{below}
	\nlabel{b}{\textcolor{blue}{7}}{left}		
\end{tikzpicture}}

\newcommand{\TSevenPlus}{%
\begin{tikzpicture}[scale=0.85]
  \DrawBaseTwoDiamonds
  \draw[tedge] (l1)--(l2);
  \draw[bend left=25] (t) to (b);
	\nlabel{m1}{\textcolor{brown}{22}}{above}
	\nlabel{m2}{\textcolor{brown}{22}}{below}
	\nlabel{m1}{\textcolor{blue}{10}}{left}
	\nlabel{m2}{\textcolor{blue}{10}}{left}
	\nlabel{t}{\textcolor{blue}{10}}{left}	
	\nlabel{l1}{\textcolor{blue}{2}}{above}
	\nlabel{l2}{\textcolor{blue}{2}}{below}
	\nlabel{b}{\textcolor{blue}{10}}{left}	
\end{tikzpicture}}

\newcommand{\TEightPlus}{%
\begin{tikzpicture}[scale=0.85]
  \DrawBaseTwoDiamonds
  \draw[tedge] (l1)--(b);
  \draw[tedge] (t)--(l2);
	\nlabel{m1}{\textcolor{brown}{23}}{above}
	\nlabel{m2}{\textcolor{brown}{23}}{below}
	\nlabel{m1}{\textcolor{blue}{10}}{right}
	\nlabel{m2}{\textcolor{blue}{10}}{right}
	\nlabel{t}{\textcolor{blue}{7}}{left}	
	\nlabel{l1}{\textcolor{blue}{6}}{above}
	\nlabel{l2}{\textcolor{blue}{6}}{below}
	\nlabel{b}{\textcolor{blue}{7}}{left}	
\end{tikzpicture}}

\newcommand{\TNinePlus}{%
\begin{tikzpicture}[scale=0.85]
  \DrawBaseTwoDiamonds
  \draw[tedge] (l1)--(l2);
  \draw[tedge] (r1)--(r2);
  \draw[bend left=25] (t) to (b);	
	\nlabel{m1}{\textcolor{brown}{24}}{above}
	\nlabel{m2}{\textcolor{brown}{24}}{below}
	\nlabel{m1}{\textcolor{blue}{10}}{left}
	\nlabel{m2}{\textcolor{blue}{10}}{left}
	\nlabel{t}{\textcolor{blue}{10}}{left}	
	\nlabel{l1}{\textcolor{blue}{2}}{above}
	\nlabel{r1}{\textcolor{blue}{2}}{above}
	\nlabel{l2}{\textcolor{blue}{2}}{below}
	\nlabel{r2}{\textcolor{blue}{2}}{below}
	\nlabel{b}{\textcolor{blue}{10}}{left}			
\end{tikzpicture}}

\newcommand{\TTenPlus}{%
\begin{tikzpicture}[scale=0.85]
  \DrawBaseTwoDiamonds
  \draw[tedge] (l1)--(l2);
  \draw[tedge] (t)--(r2);
  \draw[tedge] (b)--(r1);
	\nlabel{m1}{\textcolor{brown}{25}}{above}
	\nlabel{m2}{\textcolor{brown}{25}}{below}
	\nlabel{m1}{\textcolor{blue}{10}}{left}
	\nlabel{m2}{\textcolor{blue}{10}}{left}
	\nlabel{t}{\textcolor{blue}{7}}{left}	
	\nlabel{l1}{\textcolor{blue}{2}}{above}
	\nlabel{r1}{\textcolor{blue}{6}}{above}
	\nlabel{l2}{\textcolor{blue}{2}}{below}
	\nlabel{r2}{\textcolor{blue}{6}}{below}
	\nlabel{b}{\textcolor{blue}{7}}{left}		
\end{tikzpicture}}


\begin{figure}[t]
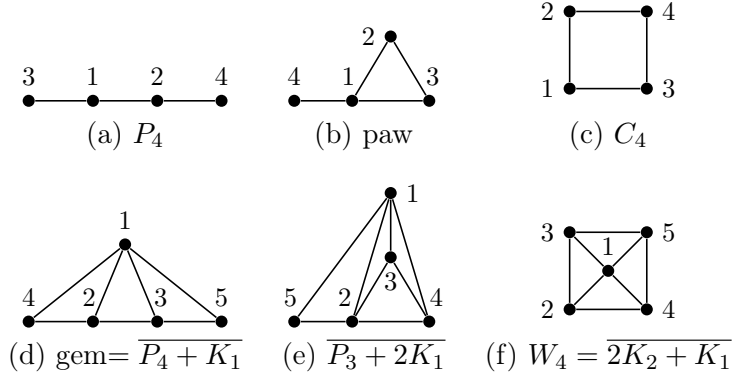

\centering
\begin{tabular}{ccc}
\PFour  & \Paw  & \CFour \\
(a) $P_4$  & (b) paw  & (c) $C_4$ \\
\\[-2mm]
\PFourUniversal  & \PawUniversal  & \CFourUniversal \\
(d) gem$=\overline{P_4+K_1}$  & (e) $\overline{P_3+2K_1}$  & (f) $W_4=\overline{2K_2+K_1}$
\end{tabular}
	\caption{Some graphs equipped with a degree--lexicographic ordering.
The last three are exactly the forbidden induced subgraphs for LUCS graphs and
also belong to the family of forbidden induced subgraphs for probe diamond-free graphs.
For each of these graphs, all degree--lexicographic orderings are equivalent.}
	\label{fig: forbidden  subgraphs 3}
\end{figure}

\begin{figure}[t]
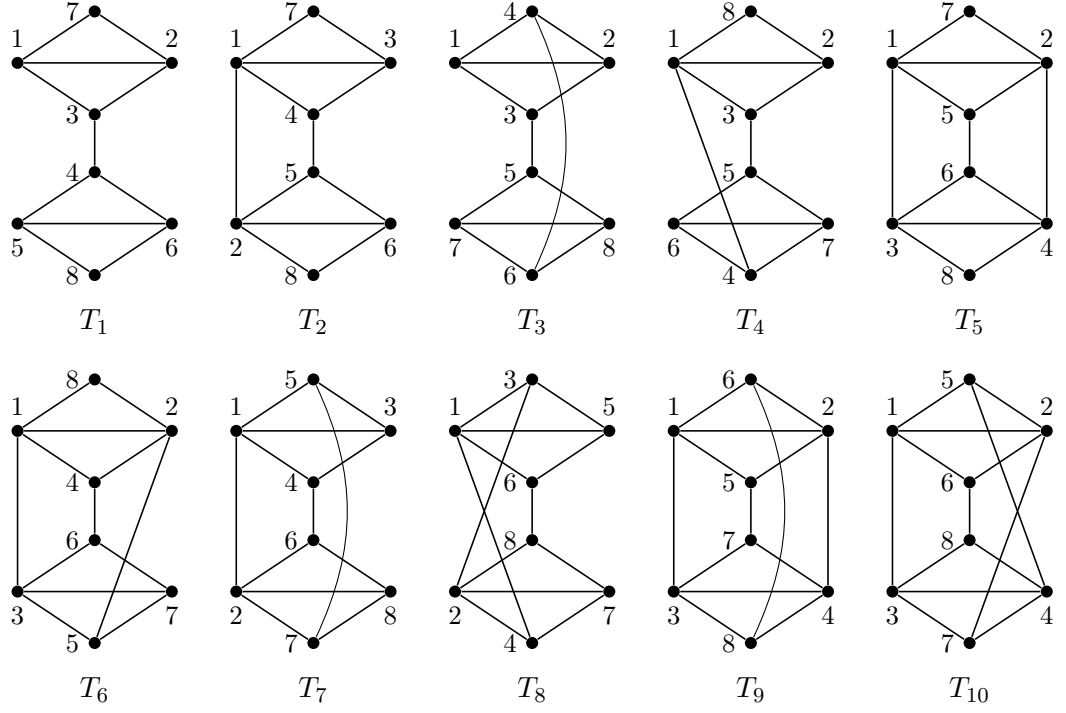

\centering
\begin{tabular}{ccccc}
\TOne  & \TTwo  & \TThree  & \TFour & \TFive \\
$T_1$  & $T_2$  & $T_3$ & $T_4$  & $T_5$\\
\\[-2mm]
\TSix & \TSeven & \TEight & \TNine & \TTen \\
$T_6$  & $T_7$  & $T_8$  & $T_9$  & $T_{10}$

\end{tabular}
	\caption{The forbidden subgraphs $T_i$ for probe diamond-free graphs, each endowed with a degree--lexicographic ordering.}
	\label{fig: forbidden  subgraphs}
\end{figure}

\begin{definition}\label{def: bip aux}
	Let \textcolor{black}{$\phi:V(G)\to \{0,\dots,|V(G)|-1\}$ be a bijection, where $G$ is a LUCS graph}. We define the auxiliary bipartite graph \textcolor{black}{
\[
G^\phi_B=(A^\phi_B\cup N^\phi_G,E^\phi_B)
\]
as follows. Initially, $A^\phi_B=N^\phi_G=E^\phi_B=\emptyset$.
}	
For each vertex $v \in V(G)$ and each non-special component \textcolor{black}{$C_v$} of $G[N(v)]$ such that $\phi(v)<\phi(x)$ for every $x\in K_{C_v}$\textcolor{black}{, let} $w\in K_{C_v}$ \textcolor{black}{be} the vertex with minimum \textcolor{black}{$\phi$-value. Then}:
	\begin{itemize}
		\item \textcolor{black}{add} a vertex $a_{vw}$ to \textcolor{black}{$A^\phi_B$.}
		\item \textcolor{black}{set $N^\phi_G\leftarrow N^\phi_G\cup S_{C_v}$.}
		\item \textcolor{black}{add edges joining $a_{vw}$ to every vertex $s\in S_{C_v}$}.
	\end{itemize}
\end{definition}

The following observation \textcolor{black}{follows directly} from the definition.
\begin{remark}~\label{rem: pi}
Let $G$ be a \textcolor{black}{LUCS} graph \textcolor{black}{and let $v\in V(G)$}. If \textcolor{black}{$C_v$} is a non-special component of $G[N(v)]$, then \textcolor{black}{the clique $K_{C_v}\cup\{v\}$} is represented by exactly one vertex \textcolor{black}{$a_{uz}$ in $G^\phi_B$, where $u$ and $z$ are respectively the vertices with the minimum and second minimum $\phi$-values in $K_{C_v}\cup\{v\}$, for any bijection $\phi$.}

\textcolor{black}{Moreover:
\begin{itemize}
\item $(K_{C_v}\cup\{v\}\setminus\{u\},S_{C_v})$ is the unique complete split partition of the connected component of $G[N(u)]$ that contains $z$.
\item $N_G=N^\phi_G$.
\item Two vertices $x,y$ are tips of some induced diamond of $G$ if and only if their distance in $G^\phi_B$ is exactly $2$.
\item For any pair of bijections $\phi$ and $\phi'$, the graphs $G^\phi_B$ and $G^{\phi'}_B$ are isomorphic. Hence, we may simply denote them by $G_B$.
\end{itemize}
}
\end{remark}

\begin{proof}
\textcolor{black}{
Fix a bijection $\phi:V(G)\to\{0,\dots,|V(G)|-1\}$ and consider the construction
of $G_B^\phi$ in Definition~\ref{def: bip aux}.
Let $v\in V(G)$ and let $C_v$ be a non-special connected component of $G[N(v)]$.
Since $C_v$ is a connected complete split graph and $|S_{C_v}|\ge 2$, it has a
unique complete split partition $(K_{C_v},S_{C_v})$ and, moreover, $|K_{C_v}|\ge 1$.
Consequently, the graph $G[K_{C_v}\cup S_{C_v}\cup\{v\}]$ is also complete split
with unique complete split partition $(K_{C_v}\cup\{v\},\,S_{C_v})$, and
$|K_{C_v}\cup\{v\}|\ge 2$.
}

\textcolor{black}{
Let $u$ and $z$ be respectively the vertices with minimum and second minimum
$\phi$-values in the clique $K_{C_v}\cup\{v\}$.
Let $D$ be the connected component of $G[N(u)]$ that contains $z$.
Since $G$ is LUCS, $D$ is complete split; let $(K_D,S_D)$ denote its (unique)
complete split partition.
}

\textcolor{black}{
We first prove that $S_{C_v}\subseteq V(D)$.
Indeed, we have $S_{C_v}\subseteq N(u)$ because $u\in K_{C_v}\cup\{v\}$ is complete
to $S_{C_v}$, and similarly $S_{C_v}\subseteq N(z)$ because
$z\in K_{C_v}\cup\{v\}$.
Hence every vertex of $S_{C_v}$ is adjacent to both $u$ and $z$, and therefore is
adjacent to $z$ in $G[N(u)]$.
Thus $S_{C_v}\subseteq V(D)$.
Since $|S_{C_v}|\ge 2$ and $S_{C_v}$ is independent, $D$ is non-special and hence
$(K_D,S_D)$ is unique.
}

\textcolor{black}{
Next, we show that
\[
(K_{C_v}\cup\{v\})\setminus\{u\}\subseteq V(D).
\]
Indeed, every vertex
$w\in (K_{C_v}\cup\{v\})\setminus\{u,z\}$
is adjacent to $u$ and to $z$, because $K_{C_v}\cup\{v\}$ is a clique.
Hence $w\in N(u)$ and is adjacent to $z$ in $G[N(u)]$, so $w\in V(D)$.
}

\textcolor{black}{
Moreover, every vertex $w\in (K_{C_v}\cup\{v\})\setminus\{u\}$ must belong to $K_D$:
since $w$ is adjacent to all vertices of $S_{C_v}\subseteq V(D)$ and
$|S_{C_v}|\ge 2$, it cannot lie in $S_D$.
Therefore
\[
(K_{C_v}\cup\{v\})\setminus\{u\}\subseteq K_D
\qquad\textit{and}\qquad
S_{C_v}\subseteq S_D.
\]
}

\textcolor{black}{
Finally, we prove that $V(D)$ has no additional vertices.
Suppose that there exists
\[
x\in V(D)\setminus\Bigl(((K_{C_v}\cup\{v\})\setminus\{u\})\cup S_{C_v}\Bigr).
\]
Then $x\neq u$ (since $u\notin N(u)$) and $x$ is adjacent to $z$, because
$z\in (K_{C_v}\cup\{v\})\setminus\{u\}\subseteq K_D$ and every vertex of $D$
is adjacent to all vertices of $K_D$.
}

\textcolor{black}{
We claim that $x$ is also adjacent to $v$.
If $u=v$, then $x\in N(u)=N(v)$ and the claim holds.
Otherwise, $v\in (K_{C_v}\cup\{v\})\setminus\{u\}\subseteq K_D$, so again
$x$ is adjacent to $v$.
Thus $x\in N(v)$.
Since $x$ is adjacent to $z\in V(C_v)$, the vertex $x$ lies in the same connected
component of $G[N(v)]$ as $z$, namely in $C_v$.
Therefore $x\in V(C_v)=K_{C_v}\cup S_{C_v}$, contradicting the choice of $x$.
}

\textcolor{black}{
Hence
\[
K_D=(K_{C_v}\cup\{v\})\setminus\{u\}
\qquad\textit{and}\qquad
S_D=S_{C_v}.
\]
}

\textcolor{black}{
By Definition~\ref{def: bip aux}, when processing the pair $(u,D)$ the
construction adds exactly one vertex $a_{uw}$, where $w$ is the vertex of
minimum $\phi$-value in $K_D$.
Since $z$ is the second minimum in $K_{C_v}\cup\{v\}$, we have $w=z$,
and therefore the representative is $a_{uz}$.
Uniqueness follows because such a representative is created only for the unique
minimum-$\phi$ vertex $u$ of the clique $K_{C_v}\cup\{v\}$.
}

\textcolor{black}{
\smallskip
\noindent
\emph{Item (1).}
Let $D$ be the connected component of $G[N(u)]$ that contains $z$.
Then $(K_{C_v}\cup\{v\}\setminus\{u\},\,S_{C_v})$ is the unique complete split
partition of $D$.
}

\textcolor{black}{
\smallskip
\noindent
\emph{Item (2).}
By Definition~\ref{def: bip aux}, a vertex is added to $N_G^\phi$ if and only if
it belongs to $S_{C_t}$ for some non-special component $C_t$ of $G[N(t)]$.
Equivalently, this happens if and only if it is an endpoint of a non-edge
$xy$ with $\{x,y\}\subseteq S_{C_t}$, that is, if and only if it belongs to
$N_G$ by the definition of $F_G$.
Hence $N_G^\phi=N_G$.
}

\textcolor{black}{
\smallskip
\noindent
\emph{Item (3).}
Assume that $x$ and $y$ are tips of an induced diamond in $G$.
Then there exists a vertex $t$ of degree three in that diamond such that
$x,y\in N(t)$ and $xy\notin E(G)$.
Since $G$ is LUCS, the vertices $x$ and $y$ lie in the independent side
$S_{C_t}$ of the same non-special component $C_t$ of $G[N(t)]$.
By Definition~\ref{def: bip aux}, the construction creates a vertex
$a\in A_B^\phi$ adjacent to every vertex of $S_{C_t}$, in particular to both
$x$ and $y$.
Hence $\mathrm{dist}_{G_B^\phi}(x,y)=2$.
}

\textcolor{black}{
Conversely, if $\mathrm{dist}_{G_B^\phi}(x,y)=2$, then $x$ and $y$ have a common
neighbor $a_{tw}\in A_B^\phi$.
By construction, this implies that $x,y\in S_{C_t}$ for some non-special
component $C_t$ of $G[N(t)]$, where $w\in K_{C_t}$, and that $xy\notin E(G)$.
Hence both $t$ and $w$ are adjacent to $x$ and $y$, while $x$ and $y$ are
nonadjacent.
Therefore, $G[\{t,w,x,y\}]$ is an induced diamond with tips $x$ and $y$.
}

\textcolor{black}{
\smallskip
\noindent
\emph{Item (4).}
Let $\phi$ and $\phi'$ be two bijections.
The construction of $G_B^\phi$ depends only on which vertex is the minimum and
second minimum (within each clique $K_{C_v}\cup\{v\}$) and on the associated
independent side $S_{C_v}$.
Thus changing $\phi$ only relabels the vertices of $A_B^\phi$ without modifying
the incidence relation between $A_B^\phi$ and $N_G$.
Therefore, there is a natural bijection between $A_B^\phi$ and $A_B^{\phi'}$
mapping each representative of a clique to the representative of the same
clique and fixing $N_G$ pointwise.
This bijection preserves adjacency, and hence
$G_B^\phi$ and $G_B^{\phi'}$ are isomorphic.
}
\end{proof}

\subsection{New characterization for probe diamond-free graphs}

\textcolor{black}{The following lemma is a direct consequence of Lemma~\ref{lem: complete split} and
Theorem~\ref{thm: charaterization non-partitioned}.}
\begin{lem}\label{lem: probe diamond-free implies locally complete split}
	If $G = (V, E)$ is probe diamond-free, then $G$ is a \textcolor{black}{LUCS} graph.
\end{lem}

\begin{figure}[t]
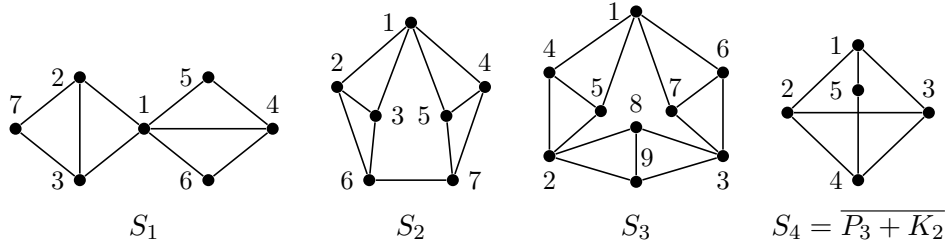

\centering
\begin{tabular}{cccc}
\SOne  & \STwo  & \SThree & \SFour \\
$S_1$  & $S_2$  & $S_3$   & $S_4=\overline{P_3 + K_2}$
\end{tabular}
\caption{The forbidden subgraphs $S_i$ for probe diamond-free graphs, whose vertices are labeled according to a degree--lexicographic ordering. Subgraphs $S_4$ and $S_3$ correspond to the unpartitioned versions of $H_4$ and $H_5$, respectively. For $S_2$ and $S_3$, all degree--lexicographic orderings are equivalent}
\label{fig: forbidden  subgraphs 2}
\end{figure}


Using Lemma~\ref{lem: complete split} and a careful inspection of the proof of~\cite[Lemma~5]{BonomoFDDGSS}, we obtain the following \textcolor{black}{result}.
\begin{lem}\label{lem: N_G stable set two} 
Let $G$ be a \textcolor{black}{LUCS} graph. If $N_G$ is not an independent set, then $G$ contains
$\overline{P_3 + K_2}$, or one of the graphs depicted in
Figure~\ref{fig: forbidden  subgraphs}, or $S_1$ or $S_2$
(see Figure~\ref{fig: forbidden  subgraphs 2}) as an induced subgraph.
\end{lem}
%
\textcolor{black}{The next} result will be useful to detect subgraphs isomorphic to $\overline{P_3 + K_2}$.
\begin{lem}\label{lem: neighborhood of vertices in N_G}
Let $G$ be a \textcolor{black}{LUCS} graph such that $N_G$ is an \textcolor{black}{independent} set. Then
$G[N(v)]$ is $P_3$-free for every vertex $v\in N_G$. 
Additionally, if \textcolor{black}{$K\subseteq V(G)\setminus N_G$} is a clique such that
\textcolor{black}{$|N(v)\cap K|\ge 2$ for some} $v\in N_G$, then $v$ is adjacent to \textcolor{black}{every vertex of $K$}.
\end{lem}
\begin{proof}
\textcolor{black}{
Suppose towards a contradiction that $G[N(v)]$ contains an induced $P_3$.
Then its three vertices together with $v$ induce a diamond $D$.
Since $v$ is adjacent to all three vertices of $D$, it is not a tip of $D$;
hence the tips of $D$ are exactly the two endpoints of the missing edge.
By Remark~\ref{rmk: tip of a dimond}, both tips belong to $N_G$, and since $v\in N_G$
and is adjacent to them, this contradicts that $N_G$ is an independent set.
}

\textcolor{black}{
Now suppose that there exists a clique $K\subseteq V(G)\setminus N_G$ with
$|N(v)\cap K|\ge 2$, and let $w\in K\setminus N(v)$.
Then $v$ and $w$ are nonadjacent, while both are adjacent to every vertex in
$N(v)\cap K$, which has size at least two.
Thus $\{v,w\}\cup (N(v)\cap K)$ contains a diamond whose missing edge is $vw$,
so $w$ is a tip of an induced diamond of $G$.
By Remark~\ref{rmk: tip of a dimond}, $w\in N_G$, contradicting
$K\subseteq V(G)\setminus N_G$.
}
\end{proof}
\begin{lem}\label{lem: C_6}
Let $G$ be a \textcolor{black}{LUCS} $\overline{P_3 +  K_2}$-free graph such that $N_G$ is an
\textcolor{black}{independent} set. \textcolor{black}{Then $G = ((V(G)\setminus N_G)\cup N_G, E(G))$ contains $H_5$ as a partitioned subgraph if and only if
$G^\phi_B$ contains an induced $C_6$, for any bijection $\phi$}. 
\end{lem}

\begin{proof}
\textcolor{black}{
Assume first that
$G=((V(G)\setminus N_G)\cup N_G,E(G))$ contains $H_5$ as a partitioned subgraph.
Let $v_1,v_2,v_3\in N_G$ be the three vertices of $N_G$ in this copy of $H_5$,
and let $u_{1,2},z_{1,2},u_{2,3},z_{2,3},u_{3,1},z_{3,1}\in V(G)\setminus N_G$
be the remaining vertices, as depicted in
Figure~\ref{fig: forbidden partitioned subgraphs}, where
\[
N(u_{i,j})\cap\{v_1,v_2,v_3\}
=
N(z_{i,j})\cap\{v_1,v_2,v_3\}
=
\{v_i,v_j\}
\quad
\]
}
\textcolor{black}{
By the structure of $H_5$, the vertices $v_i,v_j,z_{i,j}$ belong to the same
connected component $C_{i,j}$ of $G[N(u_{i,j})]$, and hence
$v_i,v_j\in S_{C_{i,j}}$.
By Definition~\ref{def: bip aux}, there exists a vertex
$a_{\tau_{i,j}\mu_{i,j}}\in A^\phi_B$
adjacent to both $v_i$ and $v_j$ in $G^\phi_B$. Let $\{v_1,v_2,v_3\}\setminus \{v_i,v_j\}=\{v_k\}$,  by Lemma~\ref{lem: neighborhood of vertices in N_G}, $a_{\tau_{i,j}\mu_{i,j}}$ is not neighbor of  
$v_k$ in $G^\phi_B$ since $\{\tau_{i,j}, \mu_{i,j}\}\not\subseteq N(v_k)$ because $u_{i,j},z_{i,j)}\notin N(v_k)$ and $\{\tau_{i,j}, \mu_{i,j}\}\cup \{u_{i,j},z_{i,j}\}$ is a clique.
Moreover, the vertices
$a_{\tau_{1,2}\mu_{1,2}},
a_{\tau_{2,3}\mu_{2,3}},
a_{\tau_{3,1}\mu_{3,1}}$
are pairwise distinct, since each of them is adjacent to a different pair
among $\{v_1,v_2,v_3\}$.
Therefore,
}
\textcolor{black}{
\[
v_1,
a_{\tau_{1,2}\mu_{1,2}},
v_2,
a_{\tau_{2,3}\mu_{2,3}},
v_3,
a_{\tau_{3,1}\mu_{3,1}}
\]
induces a $C_6$ in $G_B^\phi$.
}

\textcolor{black}{
Conversely, assume that $G_B^\phi$ contains an induced $C_6$ on
\[
\{a_{r_is_i}:1\le i\le3\}\cup\{v_i:1\le i\le3\},
\]
where $a_{r_is_i}$ is adjacent exactly to $v_i$ and $v_{i+1}$
(indices modulo $3$).
By the construction of $G_B^\phi$, for each $i$ there exists a vertex
$u_{i,i+1}\in V(G)\setminus N_G$ and a non-special component
$C_{i,i+1}$ of $G[N(u_{i,i+1})]$ such that
$\{v_i,v_{i+1}\}\subseteq S_{C_{i,i+1}}$.
Moreover, since the $C_6$ is induced, the vertex $a_{r_is_i}$ is adjacent to
exactly two vertices of $\{v_1,v_2,v_3\}$.
Hence, for the third vertex $v_k$, either
$v_k$ is not adjacent to $u_{i,i+1}$, or
$v_k$ is adjacent to $u_{i,i+1}$ but $v_k\notin C_{i,i+1}$.
In both cases, $v_k\notin S_{C_{i,i+1}}$.
Using that $N_G$ is an independent set and
Lemma~\ref{lem: neighborhood of vertices in N_G},
these vertices form a partitioned subgraph isomorphic to $H_5$.
}
\end{proof}

	%
%
\textcolor{black}{We are now ready} to prove the main theorem of this section, which underpins our recognition algorithm.
\begin{theorem}\label{thm: new characterization}
\textcolor{black}{A graph $G$} is probe diamond-free if and only if it is a \textcolor{black}{\emph{LUCS}} $\overline{P_3+K_2}$-free graph such that $N_G$ is an independent set and $G_B$ is $C_6$-free.
\end{theorem}

\begin{proof}
The ``only if'' \textcolor{black}{direction} follows from the fact that probe diamond-free graphs are $\overline{P_3+2K_1}$-free \textcolor{black}{and from} Lemmas~\ref{lem: probe diamond-free implies locally complete split}, \ref{lem: N_G stable set two}, and~\ref{lem: C_6}.

The ``if'' \textcolor{black}{direction} follows from Theorem~\ref{thm: charaterization partitioned} by using Lemma~\ref{lem: C_6} for the probe partition $(V(G)\setminus N_G,\, N_G)$.
\end{proof}
\section{$O(nm)$-recognition algorithm for probe diamond-free graphs}\label{sec: algorithm}
%
In this section we present an \(O(nm)\)-time recognition algorithm. 
%
		%
%
The subroutine $\textsc{BFS}(G,z,x,y)$ performs a breadth-first search on the
connected graph $G$ starting from $z$ and stops as soon as either $x$ or $y$
is reached, where $xy$ is an edge of $G$ and $x,y$ are non-neighbors of $z$.
Let $a\in\{x,y\}$ be the first vertex reached by the search,
let $\{b\}=\{x,y\}\setminus\{a\}$,
let $c$ be the predecessor of $a$ in the BFS tree,
and let $d$ be the predecessor of $c$.

The procedure returns two values:
(i) a number in $\{1,3\}$, and
(ii) a sequence of four vertices.
More precisely, it returns
\[
(1,\,[c,a,d,b]) \quad \textit{if } cb\notin E(G),
\]
and
\[
(3,\,[c,a,b,d]) \quad \textit{otherwise}.
\]
In the first case, the induced subgraph on $\{a,b,c,d\}$ is a $P_4$,
while in the second case it is a paw.
The sequence is sorted in non-increasing order of the vertex degrees in the
induced subgraph, and ties are broken lexicographically according to their
neighborhoods. We refer to this ordering as the \emph{degree--lexicographic
ordering}.

This subroutine is needed to implement the following $O(n+m)$-time algorithm ($\textsc{non\_complete\_split}$),
which, together with the $\textsc{CSDA}$ algorithm mentioned in
Remark~\ref{rem:complete_split_algoritm},
constitute intermediate steps toward the design of our final recognition
algorithm with a negative certificate.

\begin{algorithm}[H]
	\caption{\textsc{non\_complete\_split}}\label{algo: non local complete split}
	\begin{algorithmic}[1]
		\Require A connected non-complete split graph $G$ with a universal vertex $x$ such that $G-x$ is connected.
\Ensure An indicator number $\textsc{Ind}$ and a sequence $Q$ of five vertices of $V(G)$.
The vertices in $Q$ induce either  
(1) a $\overline{P_4 + K_1}$,  
(2) a $\overline{2K_2 + K_1}$, or  
(3) a $\overline{P_3 + 2K_1}$.  
The value of $\textsc{Ind}$ specifies which of these three cases occurs.
The sequence $Q$ is ordered according to the
\emph{degree--lexicographic ordering}.
		\State $Y\gets\emptyset $
		\For{$v\in V(G)$}
		\If{$d_G(v) < |V(G)| - 1$}
		\State add $v$ to $Y$
		\EndIf
		\EndFor
		\State Choose $ab\in E(G[Y])$
		\State Choose $a'\in Y\setminus N(a)$
		\If{$a'\notin N(b)$}
		\State $(\textsc{Ind},Q)\gets \textsc{BFS}(G-x,\, a',\, a,\, b)$
		\Else 
		\State choose $b'\in Y\setminus N(b)$
		\If{$b'\notin N(a)$}
		\State $(\textsc{Ind},Q)\gets \textsc{BFS}(G-x,\, b',\, a,\, b)$
		\Else
		\State $Q\gets [a,b,b',a']$
		\If{$a'\notin N(b')$}
		\State $\textsc{Ind}\gets 1$
		\Else
		\State $\textsc{Ind}\gets 2$
		\EndIf
		\EndIf		
		\EndIf
		\State \Return ($\textsc{Ind}$,$[x]\mathbin{\|}Q$)		
	\end{algorithmic}
\end{algorithm}
Note that the existence of the edge $ab\in E(G[Y])$ in step $7.$ is a consequence of that $G$ is not a complete split graph. Additionally, in step $8.$ ($12.$), such a vertex $a'$ ($b'$) does exist because otherwise $a\notin Y$ ($b\notin Y$).\\

The following $O(n + m)$-time algorithm ($\textsc{auxiliary\_bipartite\_graph}$) constructs the auxiliary bipartite graph described in Section~\ref{sec: characterization} (we take $\phi$ to be the identity ordering), with $O(n + m)$ vertices, and $O(m)$ edges. This graph is used in Theorem~\ref{thm: new characterization} to establish a new characterization of probe diamond-free graphs, providing an alternative to the characterization given in~\cite{BonomoFDDGSS}.
%
\begin{algorithm}[H]
	\caption{\textsc{auxiliary\_bipartite\_graph}}\label{algo: auxiliary_bipartite_graph}
	\begin{algorithmic}[1]
		\Require A locally complete split graph $G$ with $n$-vertex set $\{0,\ldots, n - 1\}$
		\Ensure The auxiliary bipartite graph $G_B$, and a map $\mathrm{rep}$ $A$.
		\State $a\gets n$
		\State $\mathrm{rep}\gets []$
		\State $V(G_B)\gets \emptyset$
		\State $I\gets \emptyset$
		\For{$0\le i\le n - 1$}
		\For{each connected component $C$ of $G[N(i)]$}
			\State $(\textsc{Answer},(K,S))\gets \textsc{CSDA}(C)$
			\State $j\gets\min_{v\in K} v$
			\If{$|S|\ge 2$ and $j> i$}
				\State $V(G_B)\gets V(G_B)\cup \{a\}\cup S$
				\State $I\gets I\cup S$
				\State add $as$ to $E(G_B)$ for each $s\in S$
				\State $\mathrm{rep}.\textsc{add}(a,(i,\, j))$
				\State $a\gets a + 1$
			\EndIf
		\EndFor
		\EndFor
		\State \Return ($G_B$, $A$)
		\end{algorithmic}
\end{algorithm}
%
Recall that $H_4$ consists of an induced diamond $D$ plus one additional vertex
$r$ of degree $2$ adjacent exactly to the tips of $D$.
In $H_4$, the only nonprobe vertices are precisely the two tips of the diamond,
that is, the endpoints of the missing edge of $D$.

Accordingly, our detection routine outputs a certificate ordered as
\[
Q=[x,u,z,y,r],
\]
where $x$ and $y$ are the tips of the induced diamond (hence $x,y\in N_G$),
$u$ and $z$ are the two degree-$3$ vertices of the diamond (true twins in the
diamond), and $r$ is the extra vertex adjacent exactly to $x$ and $y$.
Thus positions $2$ and $3$ are the true twins, positions $1$ and $4$ are the
false twins, and the last position is the unique degree-$2$ vertex of $H_4$.
This ordering is compatible with our degree--lexicographic convention: after
fixing a tip (position $1$), the remaining roles in $H_4$ are determined
uniquely up to symmetry.

\begin{algorithm}[H]
\caption{\textsc{detect\_H4}}\label{algo:detect_H4}
\begin{algorithmic}[1]
\Require A LUCS graph $G$, the set $N_G$ (independent), and the auxiliary graph
$G_B=(A_B\cup N_G,E_B)$ together with a map $\mathrm{rep}$ such that
for each $a\in A_B$, $\mathrm{rep}(a)=(u,z)$ and $a=a_{uz}$.
\Ensure \textbf{yes} and a sequence $Q=[x,u,z,y,r]$ inducing $H_4$, where $S_4$ is its unpartitioned version and $Q$ follows the \emph{degree--lexicographic ordering} shown in Figure~\ref{fig: forbidden  subgraphs 2}; or \textbf{no}.

\State Initialize arrays $\mathrm{cnt}[a]\gets 0$, $\mathrm{wit}_1[a]\gets \bot$, $\mathrm{wit}_2[a]\gets \bot$ for all $a\in A_B$.
\State Initialize an empty list \textsc{Touched}.

\For{each vertex $r\in V(G)\setminus N_G$}
    \State \textsc{Touched}$\gets\emptyset$
    \For{each $x\in N(r)\cap N_G$}
        \For{each $a\in N_{G_B}(x)$}
            \If{$\mathrm{cnt}[a]=0$}
                \State $\mathrm{cnt}[a]\gets 1$;\ $\mathrm{wit}_1[a]\gets x$;\ add $a$ to \textsc{Touched}
            \ElsIf{$\mathrm{cnt}[a]=1$}
                \State $\mathrm{cnt}[a]\gets 2$;\ $\mathrm{wit}_2[a]\gets x$
            \EndIf
        \EndFor
    \EndFor

    \For{each $a\in$ \textsc{Touched}}
        \If{$\mathrm{cnt}[a]=2$}
            \State $x\gets \mathrm{wit}_1[a]$;\ $y\gets \mathrm{wit}_2[a]$
            \State $(u,z)\gets \mathrm{rep}(a)$
            \If{$ru\notin E(G)$ \textbf{and} $rz\notin E(G)$}
                \State \Return (\textbf{yes}, $[x,u,z,y,r]$)
            \EndIf
        \EndIf
    \EndFor

    \For{each $a\in$ \textsc{Touched}}
        \State $\mathrm{cnt}[a]\gets 0$;\ $\mathrm{wit}_1[a]\gets \bot$;\ $\mathrm{wit}_2[a]\gets \bot$
    \EndFor
\EndFor
\State \Return (\textbf{no}, $\emptyset$).
\end{algorithmic}
\end{algorithm}

\begin{remark}\label{rem:detect_H4_time}
Since $|E(G_B)|=O(m)$ and $N_G\subseteq V(G)$, Algorithm~\ref{algo:detect_H4}
runs in $O(nm)$ time in the worst case.

Indeed, for each vertex $r$ considered by the outer loop, the algorithm scans
the set $N(r)\cap N_G$ and, for each $x\in N(r)\cap N_G$, it traverses the
adjacency list $N_{G_B}(x)$ in $G_B$.
Hence the total running time is
\[
O\!\left(
\sum_{r\in V(G)}
\sum_{x\in N(r)\cap N_G}
|N_{G_B}(x)|
\right)
=
O\!\left(
\sum_{x\in N_G} d_G(x)\, d_{G_B}(x)
\right).
\]
As $d_G(x)\le n-1$ for every $x$, we obtain
\[
\sum_{x\in N_G} d_G(x)\, d_{G_B}(x)
\le
(n-1)\sum_{x\in N_G} d_{G_B}(x)
\]

\[
\le
(n-1)\sum_{x\in V(G_B)} d_{G_B}(x)
=
(n-1)\cdot 2|E(G_B)|
=
O(nm).
\]
All remaining operations (maintaining counters, resetting touched vertices, and
the constant-time adjacency checks in $G$) are dominated by this bound.
\end{remark}

\medskip
\noindent\textbf{Correctness of Algorithm~\ref{algo:detect_H4}}

\begin{proof}
We show that Algorithm~\ref{algo:detect_H4} is correct by proving that it produces
no false positives and no false negatives.

\smallskip
\noindent
\textbf{Soundness (no false positives).}
Assume that the algorithm returns \textbf{yes} together with a sequence
\[
Q=[x,u,z,y,r].
\]
By construction, the algorithm finds a vertex $a\in A_B$ such that
$x,y\in N_{G_B}(a)$, where $\mathrm{rep}(a)=(u,z)$, and a vertex
$r\in V(G)\setminus N_G$ satisfying $x,y\in N_G(r)$ and
$ru,rz\notin E(G)$.

Since $x,y\in N_{G_B}(a)$, by Definition~\ref{def: bip aux} they belong to the
independent side $S_C$ of the same non-special component $C$ of $G[N(u)]$ that
contains $z$. Hence $x$ and $y$ are nonadjacent and both adjacent to $u$ and $z$.
It follows that $\{x,y,u,z\}$ induces a diamond whose tips are exactly $x$ and $y$.

Moreover, by construction, $r$ is adjacent to $x$ and $y$ and nonadjacent to
$u$ and $z$. Therefore, the induced subgraph on
$\{x,u,z,y,r\}$ is precisely a copy of $H_4$, with $x,y$ as the nonprobe tips of
the diamond and $r$ as the unique vertex of degree~$2$.
Thus, whenever the algorithm outputs \textbf{yes}, a valid induced $H_4$ is
indeed present in $G$.

\smallskip
\noindent
\textbf{Completeness (no false negatives).}
Assume now that $G$ contains an induced subgraph isomorphic to $H_4$.
Let $x,y$ be the tips of the induced diamond $D$, let $u',z'$ be the two
degree-$3$ vertices of $D$, and let $r$ be the additional vertex adjacent exactly
to $x$ and $y$. In particular, $x,y\in N_G$, $r\in V(G)\setminus N_G$, and
$ru',rz'\notin E(G)$.

Since $x$ and $y$ are tips of an induced diamond, Remark~\ref{rem: pi} implies
that $\mathrm{dist}_{G_B}(x,y)=2$, and hence there exists a vertex
$a\in A_B$ adjacent to both $x$ and $y$.
Let $\mathrm{rep}(a)=(u,z)$.

The algorithm considers the vertex $r$ in the outer loop.
Because $x,y\in N(r)\cap N_G$, during the inner loops it will traverse
the adjacency list of $x$ in $G_B$ and, in particular, it will touch the vertex
$a$; similarly, when processing $y$ it will touch $a$ again.
Therefore, at the end of these loops we have $\mathrm{cnt}[a]=2$, and the stored
witnesses $\mathrm{wit}_1[a],\mathrm{wit}_2[a]$ are two (not necessarily unique)
vertices
\[
x',y'\in N(r)\cap N_G \cap N_{G_B}(a),
\qquad x'\neq y'.
\]
In particular, the algorithm will test the pair $(a,r)$ using
$x'=\mathrm{wit}_1[a]$ and $y'=\mathrm{wit}_2[a]$.

If $ru\notin E(G)$ and $rz\notin E(G)$, then the algorithm outputs
$[x',u,z,y',r]$. As in the soundness part, $x'$ and $y'$ belong to the
independent side of the same non-special component of $G[N(u)]$ containing $z$,
so $\{x',y',u,z\}$ induces a diamond with tips $x'$ and $y'$.
Together with $r$, which is adjacent to $x'$ and $y'$ and nonadjacent to $u$ and
$z$, this yields an induced copy of $H_4$, and the algorithm returns \textbf{yes}.

Suppose, for contradiction, that one of the following cases hold. 


\begin{itemize}
\item[(a)] $r=u$.  
Then $z$ is adjacent to $r$. Since $\mathrm{rep}(a)=(u,z)$ and $a$ encodes the
clique $K_C\cup\{u\}$ of the corresponding non-special component, the two
degree-$3$ vertices $u',z'$ of the original diamond lie in $K_C\cup\{u\}$.
Hence at least one of $u'$ or $z'$ is adjacent to $r$, contradicting $ru',rz'\notin E(G)$.

\item[(b)] $r=z$.  
This case is symmetric to~(a).

\item[(c)] Both $ru\in E(G)$ and $rz\in E(G)$.  
Then $r$ belongs to the connected component $C$ of $G[N(u)]$ that contains $z$.
Since $u',z'$ belong to the clique side of $C$ (and are universal in $C$),
they must be adjacent to $r$, again contradicting that $r$ is adjacent only to
$x$ and $y$ in $H_4$.
\end{itemize}

All cases lead to a contradiction. Therefore $ru\notin E(G)$ and $rz\notin E(G)$
must hold, and the algorithm detects a copy of $H_4$ (possibly using witnesses
$x',y'$ different from $x,y$) and returns \textbf{yes}.

\smallskip
\noindent
Since the algorithm is both sound and complete, it correctly decides whether
$G$ contains an induced subgraph isomorphic to $H_4$.
\end{proof}

We emphasize that Algorithm~\ref{algo: forbidden-locaclly-split-complete-free}
takes an arbitrary input graph $G$ and first verifies whether it is
\emph{locally union of complete split} (LUCS).
If the LUCS verification fails, the algorithm immediately outputs a forbidden
induced subgraph for the LUCS property (namely $\overline{P_4+K_1}$,
$\overline{2K_2+K_1}$, or $\overline{P_3+2K_1}$), which is also forbidden for
probe diamond-free graphs.

Assume now that the LUCS test succeeds.
In this case, the algorithm attempts to construct candidate sets $N$ and $P$
for the nonprobe and probe vertices, respectively, together with a completion
set $F$, as specified in the \textsc{Output} clause.
The construction is driven by local constraints imposed by the complete split
structure of the connected neighborhoods.

Up to this point in the paper, the development focused primarily on identifying
the set of vertices that must belong to $N$, and on verifying whether this set
forms an independent set, with the probe set implicitly defined as
$P = V(G)\setminus N$.
However, from an algorithmic standpoint, violations of these local constraints
may arise for different structural reasons, and each of them must be identified
explicitly in order to produce a correct negative certificate.

In particular, local complete split constraints may force certain vertices to
act as probes, while others are forced to act as nonprobes.
It may therefore happen that a vertex is forced, by different local
considerations, to play incompatible roles.
Such a situation makes the partition invalid and certifies that $G$ cannot
admit a probe diamond-free representation consistent with these constraints.

Since the algorithm is designed to return explicit certificates, it is
necessary not only to detect that the construction fails, but also to identify
the precise structural reason for the failure.
In this setting, every such role-conflict is witnessed by the presence of a
small forbidden induced subgraph.
More precisely, whenever a vertex is forced to play incompatible roles, the
graph contains an induced copy of either $S_1$ or $S_4$. Indeed, it is easy
to prove that if the graph is LUCS and $\{S_1,S_4\}$-free, then no vertex can
be simultaneously the tip of one diamond and a nontip of another distinct
diamond.  
Accordingly, when such a conflict is detected, the algorithm outputs a copy of
$S_1$ or $S_4$ as a negative certificate.

Even when no role-conflict arises during the local construction, the set $N$
obtained in this phase is not guaranteed to be an independent set.
Since nonprobe vertices must form an independent set, the algorithm performs an
additional validation step.
This step is implemented by a separate routine that checks whether $G[N]$ is
edgeless.

If $N$ is independent, then the algorithm returns a positive answer with
$P=V(G)\setminus N$ and the corresponding completion set $F$. Otherwise, since the graph is $S_1$-free, the existence of an edge inside
$N$ yields an induced copy of $S_2$, an induced copy of $S_4$
(see Lemma~\ref{lem: N_G stable set two} and notice that
$\overline{P_3+K_2}=S_4$), or an induced copy of one of the forbidden graphs
$T_i$ depicted in Figure~\ref{fig: forbidden  subgraphs}.

Figure~\ref{fig: Tis with t and rho} illustrates the graphs $T_i$ together with
the corresponding $t$-values of vertices and $\rho$-values of the two diamonds,
which are used by Algorithm~\ref{algo:check_N_independent} to identify the exact
obstruction.

In all negative cases, the returned induced subgraph is forbidden for probe
diamond-free graphs and serves as a certificate.

\begin{figure}
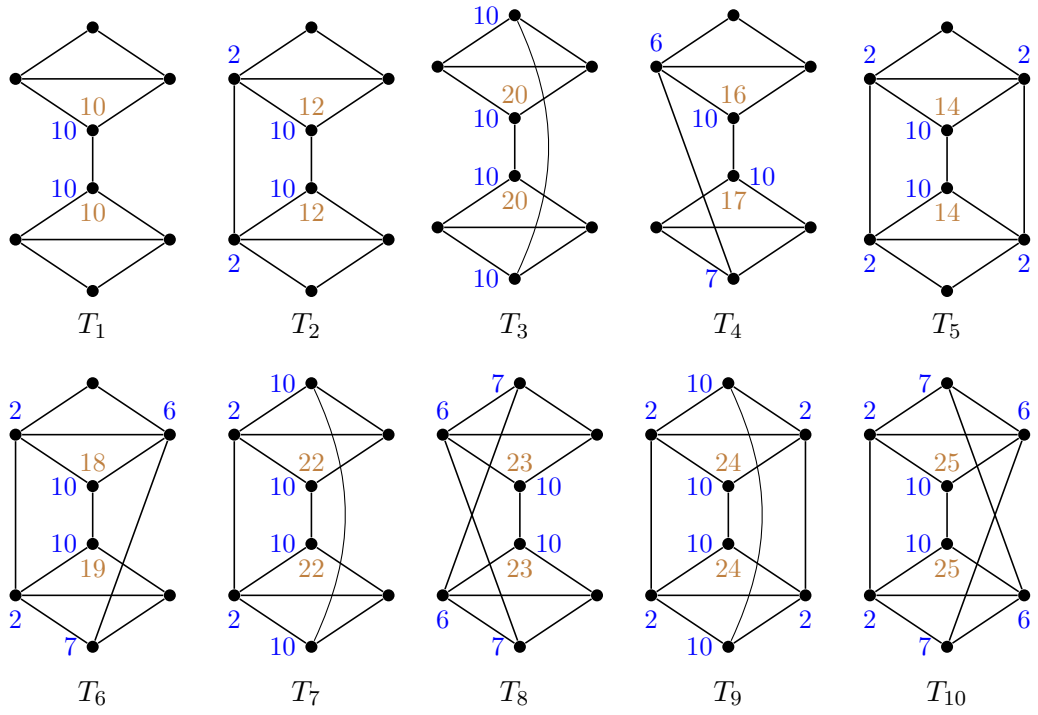

\centering
\begin{tabular}{ccccc}
\TOnePlus  & \TTwoPlus  & \TThreePlus  & \TFourPlus & \TFivePlus \\
$T_1$  & $T_2$  & $T_3$ & $T_4$  & $T_5$\\
\\[-2mm]
\TSixPlus & \TSevenPlus & \TEightPlus & \TNinePlus & \TTenPlus \\
$T_6$  & $T_7$  & $T_8$  & $T_9$  & $T_{10}$

\end{tabular}
	\caption{For each graph $T_i$, vertices with nonzero $t$-value are shown in \textcolor{blue}{blue}, and the $\rho$-values of its two diamonds are shown in \textcolor{brown}{brown}.}
	\label{fig: Tis with t and rho}
\end{figure}
\begin{algorithm}[H]
\caption{\textsc{recognizing\_LUCS\_and\_valid\_partition}}
\label{algo: forbidden-locaclly-split-complete-free}
\begin{algorithmic}[1]
\Require A graph $G$.
\Ensure \textbf{yes/no}, an indicator number $\textsc{Ind}$, a sequence $Q$,
and sets $P$, $N$, and $F$.

If the first return value is \textbf{yes}, then $P$, $N$, and $F$ are candidate
sets for the probe set, the nonprobe set, and the completion set, respectively.
Otherwise, $Q\subseteq V(G)$ induces one of the following forbidden induced
subgraphs:
\begin{enumerate}
\item[$\textsc{Ind}=1$:] a $\overline{P_4 + K_1}$;\quad
$\textsc{Ind}=2$: a $\overline{2K_2 + K_1}$;\quad
$\textsc{Ind}=3$: a $\overline{P_3 + 2K_1}$;
\item[$\textsc{Ind}=4$:] a copy of $S_1$;\quad
$\textsc{Ind}=5$: a copy of $S_4$;\quad
$\textsc{Ind}=6$: a copy of $S_2$;\quad
$\textsc{Ind}=i+6$: a copy of $T_i$ for some $1\le i\le 10$.
\end{enumerate}
In all cases, the sequence $Q$ is ordered according to the corresponding
\emph{degree--lexicographic ordering} depicted in Figures~\ref{fig: forbidden  subgraphs 3},~\ref{fig: forbidden  subgraphs} or ~\ref{fig: forbidden  subgraphs 2}.
\State Initialize arrays $\mathrm{type}[v]\gets \bot$,
$\mathrm{wit}_1[v]\gets \bot$, $\mathrm{wit}_2[v]\gets \bot$,
$\mathrm{wit}_3[v]\gets \bot$ for all $v\in V(G)$
\State $P'\gets \emptyset$; $N\gets \emptyset$; $v^\star\gets \bot$

\For{$v\in V(G)$}
  \For{each connected component $C$ of $G[N(v)]$}
    \State $(\textsc{Answer},(K,S))\gets \textsc{CSDA}(C)$
    \If{$\textsc{Answer}=\textbf{no}$}
      \State $(\textsc{Ind},Q)\gets \textsc{non\_complete\_split}(G[C\cup\{v\}],v)$
      \State \Return (\textbf{no}, $\textsc{Ind}$, $Q$, $\emptyset$, $\emptyset$, $\emptyset$)
    \Else
      \If{$|S|>1$ \textbf{and} $v^\star=\bot$}
        \For{$w\in S$}
          \If{$\mathrm{type}[w]=\bot$}
            \State add $w$ to $N$
            \State $\mathrm{type}[w]\gets \textbf{tip}$
            \State $\mathrm{wit}_1[w]\gets$ any vertex $w'\in S\setminus\{w\}$
            \State $\mathrm{wit}_2[w]\gets$ any vertex $u\in K$
            \State $\mathrm{wit}_3[w]\gets$ any vertex $z\in K\setminus\{\mathrm{wit}_2[w]\}$
          \EndIf
          \If{$\mathrm{type}[w]=\textbf{notip}$}
            \State $v^\star\gets w$
            \State $\mathrm{WIT}_1\gets$ any vertex $w'\in S\setminus\{w\}$
            \State $\mathrm{WIT}_2\gets$ any vertex $u\in K$
            \State $\mathrm{WIT}_3\gets$ any vertex $z\in K\setminus\{\mathrm{WIT}_2\}$
            \State \textsc{break}
          \EndIf
        \EndFor
				\State \textbf{if} $v^\star\ne\bot$ \textbf{then} \textsc{continue}				
        \algstore{myalg}
\end{algorithmic}
\end{algorithm}

\begin{algorithm}[H]
\ContinuedFloat
\caption{\textsc{recognizing\_LUCS\_and\_valid\_partition} (continued)}
\begin{algorithmic}[1]
\algrestore{myalg}
        \For{$w\in K$}
          \If{$\mathrm{type}[w]=\bot$}
            \State add $w$ to $P'$; $\mathrm{type}[w]\gets \textbf{notip}$
            \State $\mathrm{wit}_1[w]\gets v$; $\mathrm{wit}_2[w]\gets$ any vertex $u\in S$
            \State $\mathrm{wit}_3[w]\gets$ any vertex $z\in S\setminus\{\mathrm{wit}_2[w]\}$
          \EndIf
          \If{$\mathrm{type}[w]=\textbf{tip}$}
            \State $v^\star\gets w$; $\mathrm{WIT}_1\gets v$
            \State $\mathrm{WIT}_2\gets$ any vertex $u\in S$
            \State $\mathrm{WIT}_3\gets$ any vertex $z\in S\setminus\{\mathrm{WIT}_2\}$; \textsc{break}
          \EndIf
        \EndFor
        \State \textbf{if} $v^\star=\bot$ \textbf{then} $F\gets\{uv:\, u,v\in S\}$
      \EndIf
    \EndIf
  \EndFor
\EndFor


\If{$v^\star\ne\bot$}
  \If{$\mathrm{type}[v^\star]=\textbf{tip}$}
    \If{$N(\mathrm{wit}_1[v^\star])\cap \{\mathrm{WIT}_1,\mathrm{WIT}_2,\mathrm{WIT}_3\}\ne \emptyset$}
      \State choose $x \in N(\mathrm{wit}_1[v^\star])\cap \{\mathrm{WIT}_1,\mathrm{WIT}_2,\mathrm{WIT}_3\}$
      \State \Return (\textbf{no}, $5$, $[v^\star,\mathrm{wit}_2[v^\star],\mathrm{wit}_3[v^\star],x,\mathrm{wit}_1[v^\star]]$, $\emptyset$, $\emptyset$, $\emptyset$)
    \Else
      \State \Return (\textbf{no}, $4$, $[v^\star,\mathrm{wit}_2[v^\star],\mathrm{wit}_3[v^\star],\mathrm{WIT}_1,\mathrm{WIT}_2,\mathrm{WIT}_3,\mathrm{wit}_1[v^\star]]$, $\emptyset$, $\emptyset$, $\emptyset$)
    \EndIf
  \Else
    \If{$N(\mathrm{WIT}_1)\cap \{\mathrm{wit}_1[v^\star],\mathrm{wit}_2[v^\star],\mathrm{wit}_3[v^\star]\}\ne \emptyset$}
      \State choose $x \in N(\mathrm{WIT}_1)\cap \{\mathrm{wit}_1[v^\star],\mathrm{wit}_2[v^\star],\mathrm{wit}_3[v^\star]\}$
      \State \Return (\textbf{no}, $5$, $[v^\star,\mathrm{WIT}_2,\mathrm{WIT}_3,x,\mathrm{WIT}_1]$, $\emptyset$, $\emptyset$, $\emptyset$)
    \Else
      \State \Return (\textbf{no}, $4$, $[v^\star,\mathrm{WIT}_2,\mathrm{WIT}_3,\mathrm{wit}_1[v^\star],\mathrm{wit}_2[v^\star],\mathrm{wit}_3[v^\star],\mathrm{WIT}_1]$, $\emptyset$, $\emptyset$, $\emptyset$)
    \EndIf
  \EndIf
\EndIf
\Comment{At this point, $v^\star=\bot$ and the LUCS verification succeeded.
We now validate that $N$ is an independent set.}
\State $(\textsc{ok},\textsc{Ind},Q)\gets \textsc{check\_N\_independent}(G,N,\mathrm{type},\mathrm{wit}_1,\mathrm{wit}_2,\mathrm{wit}_3)$
\State \textbf{if} $\textsc{ok}=\textbf{no}$ \textbf{then} \Return (\textbf{no}, $\textsc{Ind}$, $Q$, $\emptyset$, $\emptyset$, $\emptyset$)
\State \Return (\textbf{yes}, $\bot$, $[]$, $V(G)\setminus N$, $N$, $F$)
\end{algorithmic}
\end{algorithm}

\begin{algorithm}[H]
\caption{\textsc{check\_N\_independent}}\label{algo:check_N_independent}
\begin{algorithmic}[1]
\Require A LUCS graph $G$, a set $N\subseteq V(G)$, and arrays
$\mathrm{type}[\cdot]$, $\mathrm{wit}_1[\cdot]$, $\mathrm{wit}_2[\cdot]$, $\mathrm{wit}_3[\cdot]$
as computed by Algorithm~\ref{algo: forbidden-locaclly-split-complete-free}.
\Ensure \textbf{yes} if $N$ is an independent set; otherwise \textbf{no}, an indicator
$\textsc{Ind}\in\{5,6,7,\dots,16\}$ and a sequence $Q$ inducing $S_4$ ($\textsc{Ind}=5$),
$S_2$ ($\textsc{Ind}=6$), or $T_i$ ($\textsc{Ind}=i+6$).

\For{each edge $vw\in E(G)$}
	\State \textbf{if} $\mathrm{type}[v]=\mathrm{type}[w]=\textbf{tip}$ \textbf{then} \Return \textsc{BuildCertificate}$(v,w)$
\EndFor
\State \Return (\textbf{yes}, $\bot$, $[]$)

\Statex

\Function{\textsc{BuildCertificate}}{$v,w$}
  \If{$\mathrm{wit}_1[v]=\mathrm{wit}_1[w]$}
    \State $Q\gets [\mathrm{wit}_1[v],\mathrm{wit}_2[v],\mathrm{wit}_3[v],\mathrm{wit}_2[w],\mathrm{wit}_3[w],v,w]$
    \State \Return (\textbf{no}, $6$, $Q$) \Comment{$S_2$}
  \EndIf

  \State $D_1\gets\{v,\mathrm{wit}_1[v],\mathrm{wit}_2[v],\mathrm{wit}_3[v]\}$
  \State $D_2\gets\{w,\mathrm{wit}_1[w],\mathrm{wit}_2[w],\mathrm{wit}_3[w]\}$

  \For{each $x\in D_1\cup D_2$}
    \State $t(x)\gets 0$
    \State $D\gets D_2$; \textbf{if} $x\in D_2$ \textbf{then} $D\gets D_1$
    \For{each $y\in D$}
      \If{$xy\in E(G)$}
        \If{$\mathrm{type}[x]=\textbf{tip}$}
          \If{$\mathrm{type}[y]=\textbf{tip}$} \State $t(x)\gets t(x)+10$
          \Else \State $t(x)\gets t(x)+7$ \EndIf
        \Else
          \If{$\mathrm{type}[y]=\textbf{tip}$} \State $t(x)\gets t(x)+6$
          \Else \State $t(x)\gets t(x)+2$ \EndIf
        \EndIf
      \EndIf
    \EndFor
  \EndFor
  \State Choose $x^\star\in D_1\cup D_2$ maximizing $t(x)$.

  \If{$t(x^\star)=14$}
    \State $Q\gets [z,\mathrm{wit}_2[v],\mathrm{wit}_3[v],\mathrm{wit}_2[w],\mathrm{wit}_3[w],v,w]$
	  \State \Return (\textbf{no}, $6$, $Q$) \Comment{$S_2$}
  \EndIf
\algstore{myalg3}
\end{algorithmic}
\end{algorithm}

\begin{algorithm}[H]
\ContinuedFloat
\caption{\textsc{BuildCertificate} (continued)}
\begin{algorithmic}[1]
\algrestore{myalg3}

  \If{$t(x^\star)=20$}
    \State $z\gets v$; \textbf{if} $x^\star\in D_1$ \textbf{then} $z\gets w$
    \State $Q\gets [z,\mathrm{wit}_2[z],\mathrm{wit}_3[z],\mathrm{wit}_1[z],x^\star]$
    \State \Return (\textbf{no}, $5$, $Q$) \Comment{$S_4$}
  \EndIf

  \State $\rho(D_1)\gets \sum_{x\in D_1} t(x)$

  \If{$\rho(D_1)=10$}
    \State $i\gets 1$
    \State $Q\gets [\mathrm{wit}_2[v],\mathrm{wit}_3[v],v,w,\mathrm{wit}_2[w],\mathrm{wit}_3[w],\mathrm{wit}_1[v],\mathrm{wit}_1[w]]$
  \ElsIf{$\rho(D_1)=12$}
    \State $i\gets 2$; \textbf{if} $t(\mathrm{wit}_2[v])=0$ \textbf{then} swap $\mathrm{wit}_2[v],\mathrm{wit}_3[v]$
		\State \textbf{if} $t(\mathrm{wit}_2[w])=0$ \textbf{then} swap $\mathrm{wit}_2[w],\mathrm{wit}_3[w]$
    \State $Q\gets [\mathrm{wit}_2[v],\mathrm{wit}_2[w],\mathrm{wit}_3[v],v,w,\mathrm{wit}_3[w],\mathrm{wit}_1[v],\mathrm{wit}_1[w]]$
  \ElsIf{$\rho(D_1)=20$}
    \State $i\gets 3$
    \State $Q\gets [\mathrm{wit}_2[v],\mathrm{wit}_3[v],v,\mathrm{wit}_1[v],w,\mathrm{wit}_1[w],\mathrm{wit}_2[w],\mathrm{wit}_3[w]]$
  \ElsIf{$\rho(D_1)=16$}
    \State $i\gets 4$ \textbf{if} $t(\mathrm{wit}_2[v])=0$ \textbf{then} swap $\mathrm{wit}_2[v],\mathrm{wit}_3[v]$
    \State $Q\gets [\mathrm{wit}_2[v],\mathrm{wit}_3[v],v,\mathrm{wit}_1[w],w,\mathrm{wit}_2[w],\mathrm{wit}_3[w],\mathrm{wit}_1[v]]$
  \ElsIf{$\rho(D_1)=17$}
    \State $i\gets 4$ \textbf{if} $t(\mathrm{wit}_2[w])=0$ \textbf{then} swap $\mathrm{wit}_2[w],\mathrm{wit}_3[w]$
    \State $Q\gets [\mathrm{wit}_2[w],\mathrm{wit}_3[w],w,\mathrm{wit}_1[v],v,\mathrm{wit}_2[v],\mathrm{wit}_3[v],\mathrm{wit}_1[w]]$
  \ElsIf{$\rho(D_1)=14$}
    \State $i\gets 5$
    \State $Q\gets [\mathrm{wit}_2[v],\mathrm{wit}_3[v],\mathrm{wit}_2[w],\mathrm{wit}_3[w],v,w,\mathrm{wit}_1[v],\mathrm{wit}_1[w]]$
  \ElsIf{$\rho(D_1)=18$}
    \State $i\gets 6$; \textbf{if} $t(\mathrm{wit}_2[v])>t(\mathrm{wit}_3[v])$ \textbf{then} swap $\mathrm{wit}_2[v],\mathrm{wit}_3[v]$
		\State \textbf{if} $t(\mathrm{wit}_2[w])=0$ \textbf{then} swap $\mathrm{wit}_2[w],\mathrm{wit}_3[w]$
    \State $Q\gets [\mathrm{wit}_2[v],\mathrm{wit}_3[v],\mathrm{wit}_2[w],v,\mathrm{wit}_1[w],w,\mathrm{wit}_3[w],\mathrm{wit}_1[v]]$
  \ElsIf{$\rho(D_1)=19$}
    \State $i\gets 6$
		\State \textbf{if} $t(\mathrm{wit}_2[v])=0$ \textbf{then} swap $\mathrm{wit}_2[v],\mathrm{wit}_3[v]$
		\State \textbf{if} $t(\mathrm{wit}_2[w])>t(\mathrm{wit}_3[w])$ \textbf{then} swap $\mathrm{wit}_2[w],\mathrm{wit}_3[w]$
    \State $Q\gets [\mathrm{wit}_2[w],\mathrm{wit}_3[w],\mathrm{wit}_2[v],w,\mathrm{wit}_1[v],v,\mathrm{wit}_3[v],\mathrm{wit}_1[w]]$
  \ElsIf{$\rho(D_1)=22$}
    \State $i\gets 7$
		\State \textbf{if} $t(\mathrm{wit}_2[v])=0$ \textbf{then} swap $\mathrm{wit}_2[v],\mathrm{wit}_3[v]$
		\State \textbf{if} $t(\mathrm{wit}_2[w])=0$ \textbf{then} swap $\mathrm{wit}_2[w],\mathrm{wit}_3[w]$
    \State $Q\gets [\mathrm{wit}_2[v],\mathrm{wit}_2[w],\mathrm{wit}_3[v],v,\mathrm{wit}_1[v],w,\mathrm{wit}_1[w],\mathrm{wit}_3[w]]$
  \ElsIf{$\rho(D_1)=23$}
    \State $i\gets 8$
		\State \textbf{if} $t(\mathrm{wit}_2[v])=0$ \textbf{then} swap $\mathrm{wit}_2[v],\mathrm{wit}_3[v]$
\algstore{myalg3}
\end{algorithmic}
\end{algorithm}

\begin{algorithm}[H]
\ContinuedFloat
\caption{\textsc{BuildCertificate} (continued number 2)}
\begin{algorithmic}[1]
\algrestore{myalg3}
		\State \textbf{if} $t(\mathrm{wit}_2[w])=0$ \textbf{then} swap $\mathrm{wit}_2[w],\mathrm{wit}_3[w]$
    \State $Q\gets [\mathrm{wit}_2[v],\mathrm{wit}_2[w],\mathrm{wit}_1[v],\mathrm{wit}_1[w],\mathrm{wit}_3[v],v,\mathrm{wit}_3[w],w]$		  \ElsIf{$\rho(D_1)=24$}
    \State $i\gets 9$
    \State $Q\gets [\mathrm{wit}_2[v],\mathrm{wit}_3[v],\mathrm{wit}_2[w],\mathrm{wit}_3[w],v,\mathrm{wit}_1[v],w,\mathrm{wit}_1[w]]$
  \Else \Comment{$\rho(D_1)=25$ under the hypotheses.}
    \State $i\gets 10$
		\State \textbf{if} $t(\mathrm{wit}_2[v])>t(\mathrm{wit}_3[v])$ \textbf{then} swap $\mathrm{wit}_2[v],\mathrm{wit}_3[v]$
		\State \textbf{if} $t(\mathrm{wit}_2[w])>t(\mathrm{wit}_3[w])$ \textbf{then} swap $\mathrm{wit}_2[w],\mathrm{wit}_3[w]$
    \State $Q\gets [\mathrm{wit}_2[v],\mathrm{wit}_3[v],\mathrm{wit}_2[w],\mathrm{wit}_3[w],\mathrm{wit}_1[v],v,\mathrm{wit}_1[w],w]$
  \EndIf

  \State \Return (\textbf{no}, $i+6$, $Q$) \Comment{$T_i$}
\EndFunction
\end{algorithmic}
\end{algorithm}

\medskip
\noindent\textbf{Correctness of the LUCS+partition routine.}
Algorithm~\ref{algo: forbidden-locaclly-split-complete-free} is a subroutine of
the full recognition procedure.
Its purpose is twofold: (i) to verify the LUCS property and, when successful,
to construct candidate sets consistent with the local complete-split structure
of the connected neighborhoods; and (ii) to validate that the constructed
candidate nonprobe set $N$ is independent.
In the negative case, the algorithm must output a \emph{precise} forbidden
induced subgraph that explains the failure of this step.

Accordingly, correctness is proved by separating the two responsibilities:
first we justify the certificates produced while enforcing LUCS and detecting
role-conflicts (yielding $\overline{P_4+K_1}$, $\overline{2K_2+K_1}$,
$\overline{P_3+2K_1}$, $S_1$, or $S_4$), and then we prove that the additional
validation step (Algorithm~\ref{algo:check_N_independent}) either confirms that
$N$ is independent or returns a forbidden induced subgraph
($S_2$, $S_4$, or some $T_i$).

\medskip
\noindent\textbf{Correctness of Algorithm
\ref{algo: forbidden-locaclly-split-complete-free}
(excluding Algorithm~\ref{algo:check_N_independent}).}
\begin{proof}
We prove that whenever Algorithm~\ref{algo: forbidden-locaclly-split-complete-free}
returns \textbf{no} \emph{before} invoking
Algorithm~\ref{algo:check_N_independent}, the returned certificate $Q$ is
correct.

\smallskip
\noindent\textbf{Cases $\textsc{Ind}\in\{1,2,3\}$.}
If the algorithm outputs \textbf{no} with $\textsc{Ind}\in\{1,2,3\}$, then it
does so by calling
$\textsc{non\_complete\_split}(G[C\cup\{v\}],v)$ on some connected component
$C$ of $G[N(v)]$.
By correctness of Algorithm~\ref{algo: non local complete split}, the returned
sequence $Q$ induces, respectively, a copy of
$\overline{P_4+K_1}$, $\overline{2K_2+K_1}$, or $\overline{P_3+2K_1}$.
Hence $Q$ is a correct negative certificate.

\smallskip
\noindent\textbf{Cases $\textsc{Ind}\in\{4,5\}$.}
Assume now that the LUCS test succeeds and that the algorithm detects a
role-conflict at some vertex $v^\star$.
During the scan of connected components of neighborhoods, the algorithm assigns
each vertex $p$ a value
$\mathrm{type}[p]\in\{\textbf{tip},\textbf{notip}\}$ consistent with an induced
diamond, and stores witnesses
$\mathrm{wit}_1[p],\mathrm{wit}_2[p],\mathrm{wit}_3[p]$ that determine such a
diamond.

A conflict is detected when $v^\star$ is encountered in a second induced
diamond where it must play the opposite role.
The witnesses of this conflicting diamond are stored in the global variables
$\mathrm{WIT}_1,\mathrm{WIT}_2,\mathrm{WIT}_3$.
From this point on, the algorithm only verifies LUCS and, since the LUCS test
succeeds, it eventually returns $\textsc{Ind}\in\{4,5\}$.

Let $D$ be the induced diamond determined by
$\{v^\star,\mathrm{wit}_1[v^\star],\mathrm{wit}_2[v^\star],\mathrm{wit}_3[v^\star]\}$
and let $D'$ be the induced diamond determined by
$\{v^\star,\mathrm{WIT}_1,\mathrm{WIT}_2,\mathrm{WIT}_3\}$.
If $\mathrm{type}[v^\star]=\textbf{notip}$, we interchange the names of $D$ and
$D'$ so that $v^\star$ is a tip of $D$ and a non-tip of $D'$.

Write $V(D)=\{v^\star,w,u,z\}$ where $v^\star w\notin E(G)$ is the missing edge,
so $v^\star$ and $w$ are the tips and $u,z$ are the degree-$3$ vertices.
Write $V(D')=\{v^\star,y,x,x'\}$ where $x x'\notin E(G)$ is the missing edge of
$D'$, so $x,x'$ are the tips and $v^\star,y$ are the degree-$3$ vertices.

\smallskip\noindent
\emph{Claim 1:} $D$ and $D'$ cannot overlap beyond $v^\star$, and each non-tip of $D$
is nonadjacent to every vertex in $D'\setminus\{v^\star\}$.

\smallskip\noindent
Suppose otherwise.
Then the set
\[
X=(V(D)\cup V(D'))\setminus\{v^\star,w\}
\]
contains at least three vertices and is connected in $G[N(v^\star)]$.
Let $C$ be the connected component of $G[N(v^\star)]$ containing $X$.
Since $G$ is LUCS, $C$ is complete split and non-special, because $X$ contains
two nonadjacent vertices.

Thus one of the degree-$3$ vertices of $D$, say $u$, belongs to the clique side
$K_C$ and is universal in $C$.
Hence every vertex of $X\setminus\{u\}$ lies in $N(u)$.
Since $v^\star$ and $z$ are adjacent to $u$ and to each other, all vertices of
$(V(D)\cup V(D'))\setminus\{u\}$ lie in the same connected component $C'$ of
$G[N(u)]$.

Again, $C'$ is complete split and non-special because it contains the two
nonadjacent vertices $x$ and $x'$.
Therefore $\{x,x'\}\subseteq S_{C'}$.
But $v^\star$ is adjacent to both $x$ and $x'$, so $v^\star\in K_{C'}$.
On the other hand, $v^\star$ is nonadjacent to $w$ and $w\in V(C')$, forcing
$v^\star\in S_{C'}$, a contradiction.

\smallskip\noindent
Therefore $V(D)\cap V(D')=\{v^\star\}$ and there is no edge from $\{u,z\}$ to
$D'\setminus\{v^\star\}$.

\smallskip\noindent
\emph{Claim 2:} the algorithm outputs a correct induced $S_1$ (when
$\textsc{Ind}=4$) or a correct induced $S_4$ (when $\textsc{Ind}=5$).

\smallskip\noindent
If $w$ has no neighbor in $D'$, then the two diamonds share exactly one vertex
$v^\star$ and have no other cross-edges.
Hence the induced subgraph on $V(D)\cup V(D')$ has exactly the ten edges of the
two diamonds and is isomorphic to $S_1$.
The algorithm returns $\textsc{Ind}=4$ together with the corresponding
degree--lexicographic ordering.

Otherwise, $w$ has a neighbor $x^\ast\in\{x,x',y\}$.
By Claim~1, $x^\ast$ has no neighbor in $\{u,z\}$.
Thus, within $G[V(D)\cup\{x^\ast\}]$, the vertex $x^\ast$ is adjacent exactly to
the two tips $v^\star$ and $w$ of $D$ and to no other vertex of $D$.
Therefore $G[V(D)\cup\{x^\ast\}]$ is isomorphic to $S_4$.
The algorithm returns $\textsc{Ind}=5$ with a correct degree--lexicographic
certificate.

In all cases, the returned sequence $Q$ induces the claimed forbidden subgraph.
\end{proof}

\medskip
\noindent\textbf{Correctness of Algorithm~\ref{algo:check_N_independent}.}
\begin{proof}
Assume that Algorithm~\ref{algo: forbidden-locaclly-split-complete-free}
has verified that $G$ is LUCS and produced a set $N$ such that
$N\cap P'=\emptyset$.
Algorithm~\ref{algo:check_N_independent} scans all edges of $G$.

If no edge has both endpoints in $N$, then $N$ is an independent set and the
algorithm correctly returns \textbf{yes} with $P=V(G)\setminus N$.
Hence we assume that there exists an edge $vw\in E(G)$ with $v,w\in N$.

Let
\[
D_1 = G[\{v,\mathrm{wit}_1[v],\mathrm{wit}_2[v],\mathrm{wit}_3[v]\}]
\;\text{and}
\]
\[
D_2 = G[\{w,\mathrm{wit}_1[w],\mathrm{wit}_2[w],\mathrm{wit}_3[w]\}]\quad
\]
be the induced diamonds stored by the witness arrays.
By construction, $v,\mathrm{wit}_1[v]$ are the tips of $D_1$,
$w,\mathrm{wit}_1[w]$ are the tips of $D_2$, and all other vertices of
$D_1\cup D_2$ belong to $P'$.
Since $N\cap P'=\emptyset$ and $vw\in E(G)$, we have
$v\notin V(D_2)$ and $w\notin V(D_1)$.

\smallskip\noindent
\emph{Claim 1:} $D_1$ and $D_2$ are either vertex-disjoint or intersect in exactly
one vertex $x$, and in the latter case
$x=\mathrm{wit}_1[v]=\mathrm{wit}_1[w]$.

\smallskip\noindent
Suppose that $D_1$ and $D_2$ share a vertex $x\in P'$.
Then $(V(D_1)\cup V(D_2))\setminus\{x\}$ lies in a single connected component of
$G[N(x)]$.
Since $v,\mathrm{wit}_1[v],w,\mathrm{wit}_1[w]$ all belong to this component and
$G$ is LUCS, they must lie in the independent side of its unique complete split
partition.
In particular, $v$ and $w$ must be nonadjacent, contradicting $vw\in E(G)$.
Hence $V(D_1)\cap V(D_2)\cap P'=\emptyset$.
If $D_1$ and $D_2$ intersect, they can only share a tip, and therefore
$\mathrm{wit}_1[v]=\mathrm{wit}_1[w]$.

\smallskip\noindent
We distinguish two cases.

\smallskip\noindent
\textbf{Case 1: $\mathrm{wit}_1[v]=\mathrm{wit}_1[w]$.}
Let $x=\mathrm{wit}_1[v]=\mathrm{wit}_1[w]$.
By Claim~1, $V(D_1)\cap V(D_2)=\{x\}$.
We show that the induced subgraph on $V(D_1)\cup V(D_2)$ is exactly $S_2$.

Assume that there exists an additional edge $ab$ besides the edges of the two
diamonds and the edge $vw$.
Then $a$ and $b$ belong to different diamonds.

If one of them is $v$ or $w$ (say $a$) and the other ($b$) is a non-tip of the
opposite diamond, then the vertices of that diamond (excluding $b$) together
with $a$ lie in a single connected component of $G[N(b)]$.
Since $G$ is LUCS, this component is complete split.
As $x$ is a tip in both diamonds, we have $vx\notin E(G)$ and $wx\notin E(G)$.
Hence $v,w,x$ must all belong to the independent side of the complete split
partition of this component.
However, $vw\in E(G)$, which is a contradiction.

If both $a$ and $b$ are non-tips of different diamonds, consider the diamond
$D$ that contains $b$.
Then the vertex $a$, together with the remaining vertices of $D$, lies in a
single connected component of $G[N(b)]$, since the common tip
$x=\mathrm{wit}_1[v]=\mathrm{wit}_1[w]$ is adjacent to all non-tip vertices of
$D$.
As $G$ is LUCS, this component admits a complete split partition.
The tip $x$ belongs to the independent side of this partition, and since $a$ is
adjacent to $x$, it follows that $a$ lies in the clique side.
Consequently, $a$ is adjacent to every vertex of the independent side,
including the other tip of $D$.
In particular, there exists an edge between $a$ and a tip of the opposite
diamond, reducing the situation to the previous case, which is impossible.

Therefore no such edge exists and the induced subgraph is exactly $S_2$.
The algorithm returns $\textsc{Ind}=6$ and a correct degree--lexicographic
certificate.

\smallskip\noindent
\textbf{Case 2: $\mathrm{wit}_1[v]\neq\mathrm{wit}_1[w]$.}
By Claim~1, $D_1$ and $D_2$ are vertex-disjoint.
Consider a vertex $z\in V(D_i)$.
We show that $z$ cannot have two neighbors in $V(D_{3-i})$ of different
diamond-types (one tip and one non-tip), nor three or more neighbors in the
opposite diamond.

Indeed, suppose that some vertex $z\in V(D_i)$ has two distinct neighbors in the
opposite diamond $D_{3-i}$, one of which is a non-tip vertex $y$ and the other
is a tip vertex $u$ of $D_{3-i}$.

Then the set $(V(D_{3-i})\cup\{z\})\setminus\{y\}$ lies in a single connected
component $C$ of $G[N(y)]$.

Since $G$ is LUCS, the component $C$ is complete split.
The two tips of $D_{3-i}$ are nonadjacent and therefore belong to the
independent side of a complete split partition of $C$.
As $u$ (a tip) is adjacent to $z$, it follows that $z$ must lie in the clique
side of $C$, and thus $z$ is adjacent to every vertex of $C$ except itself,
in particular to the other tip of $D_{3-i}$.
Consequently, $z$ is adjacent to both tips of $D_{3-i}$.

Now $z$ is a vertex of $D_i$ and hence is adjacent to both tips of $D_i$ and the another non-tip.
Therefore, the four tips $\{v,\mathrm{wit}_1[v],w,\mathrm{wit}_1[w]\}$ all belong
to $N(z)$.
Moreover, since $vw\in E(G)$, these four vertices lie in the same connected
component $C'$ of $G[N(z)]$.
In $C'$, the nonedges $v\,\mathrm{wit}_1[v]\notin E(G)$ and
$w\,\mathrm{wit}_1[w]\notin E(G)$ force $v$ and $w$ into the independent side of
a complete split partition of $C'$.
But $vw\in E(G)$, contradicting that the independent side is edgeless.

This contradiction shows that $z$ cannot have neighbors of mixed types in the
opposite diamond.

Consequently, suppose that some vertex $z\in V(D_i)$ has exactly two neighbors
$x,y$ in the opposite diamond $D_{3-i}$.
By the previous arguments, $x$ and $y$ must have the same type in $D_{3-i}$:
either both are tips or both are non-tips.
We show that, in fact, the following stronger properties hold.

\smallskip\noindent
\textbf{(I) The vertex $z$ must be a tip of $D_i$.}

Indeed, we distinguish the two possible configurations.

\smallskip\noindent
\emph{(a): $x$ and $y$ are the two tips of $D_{3-i}$.}
Suppose, for contradiction, that $z$ is a non-tip of $D_i$.
Then $z$ is adjacent to the two tips of $D_i$.
By assumption, $z$ is also adjacent to the two tips $x$ and $y$ of $D_{3-i}$.
In particular, since $v$ and $w$ are tips (of different diamonds) and
$vw\in E(G)$, the vertices $\{z,v,w\}$ induce a triangle.

Consequently, at least one of the tips $v$ or $w$ has two neighbors in the
opposite diamond of different types: one is the non-tip $z$, and the other is
its tip-neighbor in the corresponding diamond.
This contradicts the previously established constraint that no vertex can have
neighbors of mixed types in the opposite diamond.
Therefore, $z$ must be a tip of $D_i$.

\smallskip\noindent
\emph{(b) $x$ and $y$ are the two non-tips of $D_{3-i}$.}
Then $(V(D_{3-i})\cup\{z\})\setminus\{x\}$ lies in a single connected component
$C$ of $G[N(x)]$.
Both tips of $D_{3-i}$ belong to $S_C$.
Since $z$ is nonadjacent to both of these tips, it follows that $z\in S_C$ and
hence $z\in N$.
If $z$ were a non-tip of $D_i$, then $z$ would belong to $P'$, contradicting
$P'\cap N=\emptyset$, and this conflict would have been detected earlier by the
algorithm.
Therefore, $z$ must be a tip of $D_i$.

This proves~(I).

\smallskip\noindent
\textbf{(II) Exactly one of the following two situations occurs.}

\smallskip\noindent
\emph{(a) $x$ and $y$ are the two tips of $D_{3-i}$.}
In this case, since $z$ is a tip of $D_i$, it is nonadjacent to the two non-tips
of $D_{3-i}$.
Hence the vertex set $V(D_{3-i})\cup\{z\}$ induces exactly a copy of $S_4$.
Accordingly, the algorithm correctly returns $\textsc{Ind}=5$ together with the
corresponding certificate.

\smallskip\noindent
\emph{(b) $x$ and $y$ are the two non-tips of $D_{3-i}$.}
Then $z\notin\{v,w\}$, since $v$ and $w$ are tips adjacent to each other in the
opposite diamond.
In this configuration, the vertices $z$, $v$, $w$, and the four non-tip vertices
of the two diamonds (seven vertices in total) induce a copy of $S_2$.
Thus the algorithm correctly returns $\textsc{Ind}=6$ with the corresponding
certificate.

\smallskip\noindent
Finally, observe that the two cases above are exactly the situations in which
a vertex has two neighbors in the opposite diamond and yields the largest
$t$-values:
in Case~(II.a) we have $\boldsymbol{t(z)=20}$, and in Case~(II.b) we have
$\boldsymbol{t(z)=14}$.
In every other configuration, each vertex has at most one neighbor in the
opposite diamond, and therefore $\boldsymbol{t(\cdot)\le 10}$.
Hence, the edges between $D_1$ and $D_2$ form a matching that necessarily
contains the edge $vw$.
In this case, the induced subgraph on $V(D_1)\cup V(D_2)$ is isomorphic to
exactly one of the graphs $T_i$ shown in
Figure~\ref{fig: forbidden  subgraphs}.

By construction of the values $t(\cdot)$ and $\rho(\cdot)$, the value
$\rho(D_1)$ uniquely determines which graph $T_i$ occurs, while the auxiliary
$t$-values determine the correspondence between the vertices of the two
diamonds and the vertices of $T_i$.
Consequently, the algorithm correctly returns the value
$\textsc{Ind}=i+6$ together with a sequence $Q$ ordered according to the
degree--lexicographic ordering of $T_i$.

In all cases where $G[N]$ contains an edge, the algorithm outputs a correct
forbidden induced subgraph as a negative certificate.
\end{proof}

\medskip
\noindent\textbf{Conclusion for this step.}
By the previous two proofs, Algorithm~\ref{algo: forbidden-locaclly-split-complete-free}
either
(i) returns \textbf{yes} together with sets $N$ and $P=V(G)\setminus N$ such that
$N$ is independent and all constraints enforced in this step are satisfied, or
(ii) returns \textbf{no} together with a sequence $Q$ inducing a forbidden
induced subgraph (one of the LUCS obstructions, $S_1$, $S_2$, $S_4$, or some
$T_i$), which constitutes a correct negative certificate for this step.

\bigskip
We are ready to present the main result. The next algorithm uses a routine
\(\textsc{finding\_six\_cycles}(G_B)\) that can be implemented in \(O(nm)\) time:
given a bipartite graph \(G_B\) with bipartition \((A_B,N)\) and no \(4\)-cycles,
it decides whether \(G_B\) contains an induced \(C_6\) and, if so, returns a
sequence of vertices of such a cycle starting at a vertex \(s\in N\).
(Observe that in a bipartite \(C_4\)-free graph, every \(6\)-cycle is induced,
since any chord would create a \(C_4\).)

\begin{algorithm}[H]
\caption{\textsc{recognizing\_probe\_diamond\_free}}\label{alg: recognition algo}
\begin{algorithmic}[1]
\Require A graph $G$.
\Ensure \textbf{yes/no}, an indicator $\textsc{Ind}$, a sequence $Q$, and sets
$P$, $N$, and $F$.
If the first return value is \textbf{yes}, then $(P,N,F)$ is a probe/nonprobe
candidate partition with completion set $F$.
Otherwise, $Q$ induces a forbidden subgraph for probe diamond-free graphs (as
specified by $\textsc{Ind}$), ordered according to the degree--lexicographic
convention.

\State $(\textsc{Ans},\textsc{Ind},Q,P,N,F)\gets
\textsc{recognizing\_LUCS\_and\_valid\_partition}(G)$
\State \textbf{if} $\textsc{Ans}=\textbf{yes}$ \textbf{then} \Return $(\textbf{no},\textsc{Ind},Q,\emptyset,\emptyset,\emptyset)$

\State $(G_B,\mathrm{rep})\gets \textsc{auxiliary\_bipartite\_graph}(G)$
\State $(h_4,Q)\gets \textsc{detect\_H4}(G,N,G_B,\mathrm{rep})$
\State \textbf{if} $h_4=\textbf{yes}$ \textbf{then} \Return $(\textbf{no},5,Q,\emptyset,\emptyset,\emptyset)$

\State $(c_6,Q)\gets \textsc{finding\_six\_cycles}(G_B)$
\If{$c_6=\textbf{yes}$}
    \State $s\gets Q[1]$;\ $s'\gets Q[3]$;\ $s''\gets Q[5]$ \Comment{$s,s',s''\in N$}
    \State $(x_1,y_1)\gets \mathrm{rep}(Q[2])$
    \State $(x_2,y_2)\gets \mathrm{rep}(Q[6])$
    \State $(x_3,y_3)\gets \mathrm{rep}(Q[4])$
    \State \Return $(\textbf{no},17,[s,s',s'',x_1,y_1,x_2,y_2,x_3,y_3],\emptyset,\emptyset,\emptyset)$
\EndIf

\State \Return $(\textbf{yes},\bot,[],V(G)\setminus N,N,F)$
\end{algorithmic}
\end{algorithm}

\begin{remark}\label{rem:probe_diamond_free_time}
Algorithm~\ref{alg: recognition algo} runs in $O(nm)$ time.

Indeed, \textsc{recognizing\_LUCS\_and\_valid\_partition} runs in $O(nm)$ time by
Remark~\ref{rem:detect_H4_time} and the corresponding bound proved for
Algorithm~\ref{algo: forbidden-locaclly-split-complete-free} (and its call to
Algorithm~\ref{algo:check_N_independent}).

The auxiliary bipartite graph $G_B$ has $O(n+m)$ vertices and $O(m)$ edges, and
it can be constructed in $O(n+m)$ time.

Algorithm~\ref{algo:detect_H4} runs in $O(nm)$ time.
Finally, \textsc{finding\_six\_cycles}($G_B$) can be implemented by running BFS
from each $s\in N$ to compute the girth of $G_B$ restricted to even cycles; each
BFS takes $O(|E(G_B)|)=O(m)$ time, hence the total time is $O(|N|\cdot m)=O(nm)$.
All remaining operations are linear in the size of the returned certificates.
\end{remark}

\medskip
\noindent\textbf{Correctness of Algorithm~\ref{alg: recognition algo}.}
\begin{proof}
We prove that Algorithm~\ref{alg: recognition algo} is correct by showing that
it returns \textbf{yes} if and only if the input graph $G$ is probe
diamond-free, and otherwise returns a correct forbidden induced subgraph as a
certificate.

\smallskip\noindent
\textbf{Step 1: Verification of the LUCS property and local obstructions.}
The algorithm starts by invoking
Algorithm~\textsc{recognizing\_LUCS}-\textsc{\_and\_valid\_partition}.
If this procedure returns \textbf{no} with
$\textsc{Ind}\in\{1,2,3\}$, then $G$ contains an induced copy of
$\overline{P_4+K_1}$, $\overline{2K_2+K_1}$, or $\overline{P_3+2K_1}$.
Each of these graphs is forbidden for locally union of complete split graphs
and therefore also forbidden for probe diamond-free graphs.
In this case the algorithm correctly outputs a negative answer together with a
valid certificate.

Assume henceforth that the LUCS verification succeeds.
Under this assumption, the algorithm constructs a candidate nonprobe set $N$,
together with auxiliary information describing how vertices participate in
induced diamonds.
If a local role conflict is detected, the algorithm outputs an induced copy of
$S_1$ or $S_4$.
Since both graphs are forbidden for probe diamond-free graphs, this again yields
a correct negative certificate.

\smallskip\noindent
\textbf{Step 2: Validation of the nonprobe set.}
Even if no local conflict arises, the constructed set $N$ is not guaranteed to
be independent.
The call to Algorithm~\textsc{check\_N\_independent} explicitly verifies this
condition.

If $G[N]$ contains an edge, then the algorithm starts from such an edge and
returns a negative certificate.
Depending on the adjacencies between the two diamonds involved, this certificate
may be an induced copy of $S_2$, an induced copy of $S_4$, or one of the graphs
$T_i$ shown in Figure~\ref{fig: forbidden  subgraphs}.
Therefore, the algorithm correctly returns \textbf{no} in this case.

Assume therefore that after this step the algorithm has produced a valid
partition $(P,N)$ with $N$ independent and that no obstruction has been found
so far.

More precisely, the graphs $\overline{P_4+K_1}$, $\overline{2K_2+K_1}$, and
$\overline{P_3+2K_1}$ cannot occur, since they are excluded by the LUCS
verification.
Moreover, the presence of an induced $S_1$ would force a vertex to be assigned
simultaneously to the tentative sets $P'$ and $N$, contradicting the successful
construction with $P'\cap N=\emptyset$.
Finally, the presence of an induced $S_2$ or of one of the graphs $T_i$ would
imply the existence of an edge inside $N$, contradicting the fact that
$N$ is independent.

At this point, we only know that $N$ is an independent set and that these
obstructions have been excluded; an induced
$S_4$ may still exist in $G$ and will be handled in the next
step.

\smallskip\noindent
\textbf{Step 3: Reduction of $S_4$ to $H_4$ and detection of $H_4$.}
We show that any induced $S_4$ in $G$ yields an induced copy of $H_4$.

Consider an induced copy of $S_4$ in $G$.
By definition of $S_4$, it contains an induced diamond $D$ whose tips are the
two nonadjacent vertices, and it has one additional vertex $r$ of degree $2$
adjacent exactly to these two tips (and to no other vertex of $D$).

In our setting, the two tips of $D$ are precisely the nonprobe vertices, hence
they belong to $N$.
Since after Step~2 the set $N$ is an independent set, every other vertex of
this $S_4$---in particular the two non-tips of $D$ and the extra vertex $r$,
which are all adjacent to both tips---must lie in $V(G)\setminus N$.

Therefore, this induced $S_4$ determines an induced copy of $H_4$ in $G$:
a diamond whose tips lie in $N$ together with a vertex $r\in V(G)\setminus N$
adjacent exactly to the two tips.
Accordingly, the algorithm invokes \textsc{detect\_H4}.
If an induced copy of $H_4$ is found, the algorithm correctly outputs a
negative certificate.
If \textsc{detect\_H4} returns \textbf{no}, then $G$ contains no induced copy of
$H_4$, and hence no induced copy of $S_4$.

\smallskip\noindent
\textbf{Step 4: A structural property of representative vertices and absence of $C_4$ in $G_B$.}
Recall that $N$ is an independent set after Step~2, and Step~3 guarantees that
$G$ contains no induced copy of $S_4=\overline{P_3+K_2}$.

\smallskip\noindent
\emph{(a) Representatives lie in $P$.}
Every representative vertex belongs to $V(G)\setminus N$.
Indeed, each representative is adjacent to some vertex $s\in N$ (a tip of an
induced diamond), and since $N$ is independent, no vertex adjacent to $s$ may
belong to $N$.

\smallskip\noindent
\emph{(b) The set $T(s)$ induces a disjoint union of edges.}
Fix $s\in N$ and define
\[
T(s)=\{\, r\in V(G)\setminus N:\ r \text{ is a representative of some }
a\in N_{G_B}(s)\,\}.
\]
By~(a), we have $T(s)\subseteq P$.
Moreover, by Lemma~\ref{lem: neighborhood of vertices in N_G}, the graph
$G[N(s)]$ is a cluster graph, hence each connected component of $G[T(s)]$ is a
clique.
Since vertices in the same connected component of $G[T(s)]$ form a clique, they
must correspond to a single representative pair in the construction of $G_B$.
Therefore, each connected component of $G[T(s)]$ contains exactly two vertices,
and $G[T(s)]$ is a disjoint union of $K_2$'s.

\smallskip\noindent
\emph{Absence of induced $C_4$ in $G_B$.}
Suppose for contradiction that $G_B$ contains an induced $4$-cycle
$(s_1,a_1,s_2,a_2)$ with $s_1,s_2\in N$ and $a_1,a_2\in A_B$.
Let $\mathrm{rep}(a_1)=(v_1,w_1)$ and $\mathrm{rep}(a_2)=(v_2,w_2)$.
Since $a_1,a_2\in N_{G_B}(s_1)$, we have
\[
\{v_1,w_1,v_2,w_2\}\subseteq T(s_1).
\]
By~(b), $G[T(s_1)]$ is a disjoint union of edges, so
$\{v_1,w_1,v_2,w_2\}$ induces exactly $2K_2$ in $G$.

On the other hand, by definition of the edges of $G_B$, each of
$v_1,w_1,v_2$ is adjacent to both $s_1$ and $s_2$, while $s_1$ and $s_2$ are
nonadjacent because they belong to $N$.
Hence the vertex set $\{s_1,s_2,v_1,w_1,v_2\}$ induces exactly a copy of
$S_4=\overline{P_3+K_2}$ in $G$, contradicting the conclusion of Step~3.
Therefore, $G_B$ is $C_4$-free.

\smallskip\noindent
\textbf{Step 5: Detection of $C_6$ and correctness of the $S_3$ certificate.}
Assume that Algorithm~\textsc{finding\_six\_cycles} returns an induced cycle
\[
(s,a_1,s',a_2,s'',a_3)
\]
of length~$6$ in $G_B$, where $s,s',s''\in N$ and each $a_i\in A_B$ has
representative $\mathrm{rep}(a_i)=(x_i,y_i)$.

By construction of the auxiliary graph, each pair $x_i,y_i$ is adjacent in $G$.
Moreover, by Step~4(b), for any fixed $s\in N$ the set $T(s)$ induces a disjoint
union of edges.

Since $(s,a_1,s',a_2,s'',a_3)$ is an induced $C_6$ in $G_B$, each vertex among
$s,s',s''$ is adjacent in $G_B$ to exactly two of $a_1,a_2,a_3$, and hence any
two of the three vertices $a_1,a_2,a_3$ share a common neighbor in
$\{s,s',s''\}$.
Fix $i\neq j$ and let $s^\circ\in\{s,s',s''\}$ be a common neighbor of $a_i$ and
$a_j$ in $G_B$.
Then $\{x_i,y_i,x_j,y_j\}\subseteq T(s^\circ)$, and Step~4(b) implies that
$G[T(s^\circ)]$ is a disjoint union of edges.
Therefore, $\{x_i,y_i\}$ and $\{x_j,y_j\}$ must be vertex-disjoint and there is
no edge between them, otherwise the four vertices would lie in a single clique
component of $G[T(s^\circ)]$ and would correspond to a single representative
pair in $G_B$, contradicting $a_i\neq a_j$.

Applying this to the three pairs, we conclude that
\[
\{x_1,y_1,x_2,y_2,x_3,y_3\}
\]
induces exactly a copy of $3K_2$ in $G$.

Consequently, the vertex set
\[
\{s,s',s'',x_1,y_1,x_2,y_2,x_3,y_3\}
\]
induces a graph consisting of three disjoint edges on the representatives,
together with adjacencies prescribed by the cycle in $G_B$.

It remains to exclude adjacencies between vertices of $N$ and representatives
that are not prescribed by the cycle.
Suppose that there exists an edge $s^*u_i\in E(G)$, where
$s^*\in\{s,s',s''\}$ and $u_i\in\{x_i,y_i\}$ is a representative of
$a_i\notin N_{G_B}(s^*)$.
Choose $s^+\in\{s,s',s''\}\setminus\{s^*\}$ and let
$a_j\in N_{G_B}(s^*)\cap N_{G_B}(s^+)\cap\{a_1,a_2,a_3\}$.
Then the subgraph induced by
\[
\{s^*,s^+,x_j,y_j,u_i\}
\]
is isomorphic to $S_4=\overline{P_3+K_2}$, contradicting the conclusion of
Step~3.

Therefore, no such additional adjacencies exist, and the vertices
\[
\{s,s',s'',x_1,y_1,x_2,y_2,x_3,y_3\}
\]
induce exactly a copy of $S_3$.
Hence, the algorithm correctly returns this set as a forbidden induced subgraph.

\smallskip\noindent
\textbf{Conclusion.}
If the algorithm returns \textbf{yes}, then all forbidden induced subgraphs for
probe diamond-free graphs have been excluded, and the returned partition
$(P,N)$ together with $F$ yields a valid probe diamond-free representation of
$G$.
Conversely, if $G$ is not probe diamond-free, the algorithm returns
\textbf{no} together with a correct forbidden induced subgraph.
This proves the correctness of Algorithm~\ref{alg: recognition algo}.
\end{proof}

\section{Conclusions}\label{sec: conclusion} 
We studied the recognition problem for probe diamond-free graphs.
Although this class admits a characterization by forbidden induced subgraphs,
such characterizations do not, by themselves, yield efficient recognition
algorithms.

In contrast, we introduced a new local structural characterization based on the
locally union of complete split property and an auxiliary bipartite graph,
and leveraged it to design an $O(nm)$-time recognition algorithm.

A key aspect of our algorithm is that it is \emph{certificate-producing}.
For non-members, it outputs a negative certificate in the form of a
degree--lexicographically ordered sequence of vertices inducing a minimal
forbidden subgraph.
This ordered representation enables particularly simple and efficient
certificate verification, in contrast to unordered certificates, whose
verification may require checking exponentially many vertex permutations in
the maximum size of a minimal forbidden induced subgraph for the class.

For members of the class, the algorithm outputs a positive certificate given by
a probe partition $(P,N)$ together with a completion set $F$.
Although verifying such certificates can be done via diamond-free recognition,
the current best algorithms for that task~\cite{LinSS2012} may lead to higher
worst-case complexity due to the size of $F$.
This highlights an uncommon situation in which producing certificates is
asymptotically easier than verifying them.

Our algorithm computes a set $N$ and a completion set $F$ with the following
canonical properties. The set $N$ is contained in the set of nonprobes of
every admissible partition of a probe diamond-free graph, and $F$ is contained
in every completion set whose addition yields a diamond-free completion.
Consequently, $F$ has minimum cardinality among all such completion sets. In
particular, our results provide an alternative to both the graph sandwich
approach of Dantas et al.~\cite{DantasFdaST2011} and the purely
forbidden-subgraph-based recognition derived from~\cite{BonomoFDDGSS}.


\begin{thebibliography}{99}
\bibitem{BAyerLR2009}
Daniel Bayer, Van~Bang Le, and H.~N. de~Ridder.
\newblock Probe threshold and probe trivially perfect graphs.
\newblock {\em Theoret. Comput. Sci.}, 410(47-49):4812--4822, 2009.

\bibitem{BerryGL2007}
Anne Berry, Martin~Charles Golumbic, and Marina Lipshteyn.
\newblock Recognizing chordal probe graphs and cycle-bicolorable graphs.
\newblock {\em SIAM J. Discrete Math.}, 21(3):573--591, 2007.

\bibitem{BonomoFDDGSS}
Flavia Bonomo, Celina M.~H. de~Figueiredo, Guillermo Dur\'{a}n, Luciano~N.
Grippo, Mart\'{\i}n~D. Safe, and Jayme~L. Szwarcfiter.
\newblock On probe 2-clique graphs and probe diamond-free graphs.
\newblock {\em Discrete Math. Theor. Comput. Sci.}, 17(1):187--199, 2015.

\bibitem{ChandlerCKLP2008}
David~B. Chandler, Maw-Shang Chang, Ton Kloks, Jiping Liu, and Sheng-Lung Peng.
\newblock Partitioned probe comparability graphs.
\newblock {\em Theoret. Comput. Sci.}, 396(1-3):212--222, 2008.

\bibitem{ChandlerCMK2009}
David~B. Chandler, Maw-Shang Chang, Ton Kloks, Jiping Liu, and Sheng-Lung Peng.
\newblock On probe permutation graphs.
\newblock {\em Discrete Appl. Math.}, 157(12):2611--2619, 2009.

\bibitem{ChangKloksLiuPeng2005}
G.~J. Chang, T.~Kloks, J.~Liu, and S.-L. Peng.
\newblock The {PIGs} full monty -- a floor show of minimal separators.
\newblock In V.~Diekert and B.~Durand, editors, {\em Proceedings of the 22nd
	International Symposium on Theoretical Aspects of Computer Science (STACS
	2005)}, volume 3403 of {\em Lecture Notes in Computer Science}, pages
521--532, Heidelberg, 2005. Springer.

\bibitem{Chang2013}
Maw-Shang Chang, Ling-Ju Hung, and Peter Rossmanith.
\newblock Recognition of probe distance-hereditary graphs.
\newblock {\em Discrete Appl. Math.}, 161(3):336--348, 2013.

\bibitem{DabrowskiEtAl2026}
Konrad~K. Dabrowski, Tala Eagling{-}Vose, Matthew Johnson, Giacomo Paesani, and
Dani{\"{e}}l Paulusma.
\newblock Finding d-cuts in probe h-free graphs.
\newblock In Artur Jez and Jan Otop, editors, {\em Fundamentals of Computation
	Theory - 25th International Symposium, {FCT} 2025, Wroc{\l}aw, Poland,
	September 15-17, 2025, Proceedings}, volume 16106 of {\em Lecture Notes in
	Computer Science}, pages 109--121. Springer, 2025.

\bibitem{DantasFdaST2011}
Simone Dantas, Celina M.~H. de~Figueiredo, Murilo V.~G. da~Silva, and Rafael~B.
Teixeira.
\newblock On the forbidden induced subgraph sandwich problem.
\newblock {\em Discrete Appl. Math.}, 159(16):1717--1725, 2011.

\bibitem{LeR2007bis}
Van~Bang Le and H.~N. de~Ridder.
\newblock Characterisations and linear-time recognition of probe cographs.
\newblock In Andreas Brandst{\"{a}}dt, Dieter Kratsch, and Haiko M{\"{u}}ller,
editors, {\em Graph-Theoretic Concepts in Computer Science, 33rd
	International Workshop, {WG} 2007, Dornburg, Germany, June 21-23, 2007.
	Revised Papers}, volume 4769 of {\em Lecture Notes in Computer Science},
pages 226--237. Springer, 2007.

\bibitem{LeR2007}
Van~Bang Le and H.~N. de~Ridder.
\newblock Probe split graphs.
\newblock {\em Discrete Math. Theor. Comput. Sci.}, 9(1):207--238, 2007.

\bibitem{LeP2015}
Van~Bang Le and Sheng-Lung Peng.
\newblock Characterizing and recognizing probe block graphs.
\newblock {\em Theoret. Comput. Sci.}, 568:97--102, 2015.

\bibitem{LinSS2012}
Min~Chih Lin, Francisco~J. Soulignac, and Jayme~L. Szwarcfiter.
\newblock Arboricity, {$h$}-index, and dynamic algorithms.
\newblock {\em Theoret. Comput. Sci.}, 426/427:75--90, 2012.

\bibitem{McConnellN2009}
Ross~M. McConnell and Yahav Nussbaum.
\newblock Linear-time recognition of probe interval graphs.
\newblock In {\em Algorithms---{ESA} 2009}, volume 5757 of {\em Lecture Notes
	in Comput. Sci.}, pages 349--360. Springer, Berlin, 2009.

\bibitem{Nussbaum2014}
Yahav Nussbaum.
\newblock Recognition of probe proper interval graphs.
\newblock {\em Discrete Appl. Math.}, 167:228--238, 2014.

\bibitem{ZhangSchonFischerEtAl1994}
P.~Zhang, E.~Sch{\"o}n, S.~Fischer, E.~C., J.~Weiss, S.~Kistler, and P.~Bourne.
\newblock An algorithm based on graph theory for the assembly of contigs in
physical mapping of {DNA}.
\newblock {\em Computer Applications in the Biosciences}, 10(3):309--317, 1994.

\end{thebibliography}

\end{document}